 \documentclass[1p]{elsarticle}
 
 \usepackage[latin1]{inputenc}

\usepackage{graphicx}

\journal{Journal of \LaTeX\ Templates}

\usepackage{amssymb}
\usepackage{eucal}
\newcommand{\R}{{\Bbb R}}

\newcommand{\N}{{\Bbb N}}
\newcommand{\C}{{\Bbb C}}

\let \be=\beta

\let \var=\varphi
\let \vare=\varepsilon

\let \th=\theta
\let \la=\lambda

\let \p=\partial
\let \q=\quad

\usepackage{color}

\usepackage{amsmath}

\newtheorem{thm}{Theorem}
\newtheorem{lem}[thm]{Lemma}
\newtheorem{cor}[thm]{Corollary}

\newtheorem{definition}[thm]{Definition}
\newtheorem{remark}[thm]{Remark}
\newtheorem{example}[thm]{Example}
\newproof{proof}{Proof}

\begin{document}
\begin{frontmatter}

\title{On the geometric diversity of wavefronts  for the scalar Kolmogorov ecological equation }

\author[a]{Karel Has\'ik}
\author[a]{Jana Kopfov\'a}
\author[a]{Petra N\'ab\v{e}lkov\'a}
\author[c]{Sergei Trofimchuk\corref{mycorrespondingauthor}}
\cortext[mycorrespondingauthor]{Corresponding author}
\ead{trofimch@inst-mat.utalca.cl}
\address[a]{Mathematical Institute, Silesian University, 746 01 Opava, Czech Republic}
\address[c]{Instituto de Matem\'atica y F\'isica, Universidad de Talca, Casilla 747,
Talca, Chile }

\bigskip

\begin{abstract}
\noindent   In  this work,   we  answer three fundamental questions concerning monostable travelling fronts for the scalar Kolmogorov ecological equation with diffusion and spatiotemporal interaction:  these are the  questions about their existence, uniqueness and geometric shape. In the particular case of the food-limited model, we give a rigorous  proof of  the existence of  a peculiar, yet  substantive non-linearly determined  class of non-monotone and non-oscillating wavefronts. As regards to the scalar models  coming from applications, this kind of wave solutions is analyzed here  for the first time.  
\end{abstract}
\end{frontmatter}
\newpage

\section{Introduction and main results}
\subsection{Travelling waves in  the scalar Kolmogorov ecological equation }
\label{intro} 
\noindent Together with the Mackey-Glass type  diffusive equation \cite{ATV,BY,HJ,KO,LW,LLLM}
\begin{equation}\label{Mg}
u_t(t,x) = u_{xx}(t,x)  - u(t,x)+ {F((K*u )(t,x))},
\end{equation}
the scalar Kolmogorov ecological equation (cf. \cite[model (4.1)]{GC})
\begin{equation}\label{17nl}
u_t(t,x) = u_{xx}(t,x)  + u(t,x){G((K*u)(t,x))}, \quad u \geq 0,\ (t,x) \in \R^2,
\end{equation}
where  nonlocal spatiotemporal interaction among individuals  is expressed in terms  of the convolution of their density $u(t,x)\geq 0$, considered at some time $t$ and location $x$, with an appropriate  non-negative normalized kernel $K(s,y)$, 
\begin{equation}\label{ber}
(K *u)(t,x):= \int_{\R}\int_{\R_+} K(s,y) u(t-s,x-y)dsdy, \quad \int_{\R}\int_{\R_+} K(s,y)dsdy =1,  
\end{equation}
are  undoubtedly the most studied scalar reaction-diffusion models of  population dynamics.  Among others, (\ref{17nl}) includes the delayed  \cite{BS,P1,ADN,fhw,GT,GTLMS,HTa,HTb,ST,wz} and nonlocal  \cite{AC,ZAMP,BNPR,FZ,GL,HK,ST} variants of the KPP-Fisher equation,  diffusive version of F. E. Smith's \cite{Sm} 
 food-limited model \cite{GG,GC,GS,OW,TPT,WL,Pr} as well as the   single species model with Allee effect analyzed in \cite{GoL,HWF,SR,SPH}.    Recall that 
 (\ref{17nl}) possesses  the Allee effect if the maximal per capita growth $G^*:=\max_{u \geq 0} G(u)$ is reached at some  positive point, i.e.  $G(0) <G^*$, cf. \cite{Kua93}.

From both mathematical and modelling points of view, the main difference between equations (\ref{Mg}) and (\ref{17nl}) is that the term $u(t,x)$ enters (\ref{Mg})  additively while  (\ref{17nl}) multiplicatively \cite[Section 1.1]{HS}.  Precisely  the multiplicative coupling of the original function $u(t,x)$ with its transform $(K *u)(t,x)$ in (\ref{17nl}) can be considered as a  complicating factor for the analysis of this equation.

Everywhere in this paper, we  will assume the following "monostability" and continuity conditions on $G:\R_+ \to \R$:  
\begin{equation}\label{Mo}
G(u)(1-u) >0 \ \mbox{for all} \ u \geq 0, u\not=1; 
\end{equation}
$
G \in C(\R_+,\R) 
$
and  there exist finite lower one-sided derivatives
\begin{equation}\label{DG}
\underline{D}G(0^+)=\liminf_{u\to 0^+}\frac{G(u)-G(0)}{u} \in \R, \quad \underline{D}G(1^-) = \liminf_{u \to 1^{-}} \frac{G(u)}{u-1} \in \R. 
\end{equation}
Clearly, assumption (\ref{Mo}) implies that $u=0$ and $u =1$ are the only non-negative equilibria of  equation (\ref{17nl}). Besides  these equilibria, equation (\ref{17nl}) has many other bounded  solutions. 
This work is dedicated 
to the studies of the key transitory regimens,  wavefronts  and semi-wavefronts,  
connecting the trivial equilibrium and some positive (possibly inhomogeneous) steady state of (\ref{17nl}). 
 
We recall that the classical solution $u(t,x) = \phi(x +ct)$  is a
wavefront (or a travelling front) for  (\ref{17nl}) propagating with the velocity $c\geq 0$, if the profile
$\phi$ is $C^2$-smooth and  non-negative function satisfying the boundary conditions $\phi(-\infty) = 0$ and
$\phi(+\infty) = 1$. By replacing the condition $\phi(+\infty) = 1$ with the less restrictive  requirement 
$$0< \liminf_{s \to +\infty}\phi(s) \leq \limsup_{s \to +\infty}\phi(s) < \infty,$$ we obtain the definition of a semi-wavefront. 
Clearly, each wave profile $\phi$ to (\ref{17nl})  satisfies the functional differential equation  
\begin{eqnarray} \label{twe2an} 
\phi''(t) - c\phi'(t) + \phi(t){G((N_c*\phi)(t))}=0,  \quad t \in \R, \end{eqnarray}
where 
$$
N_c(s): = \int_{\R_+}K(v,s-cv)dv \geq 0, \ (N_c*\phi)(t):= \int_{\R}N_c(s) \phi(t-s)ds, \  \int_{\R}N_c(s)ds =1. 
$$
We assume that  $K(s,y)$  is  such that the  measurable function $N_c(s)$ of two arguments is well defined, so that  (\ref{twe2an}) is indeed a profile equation for  (\ref{17nl}); particularly that $N_c\in L^1(\R)$  and $N_c(s)$ depends continuously on $c$ for each fixed $s$. All these assumptions are rather weak and can be easily checked in each particular case. 
For example, if $K(s,y) = K_1(y)\delta(s),$  (the KPP-Fisher model with the nonlocal {\it spatial}  interaction) then $N_c(s)  = K_1(s)$ so that it is enough to assume that $K_1\in L^1(\R)$.  

In  this work,   we are going to answer three fundamental questions concerning the tra\-velling fronts for equation (\ref{17nl}):  these are the  questions about their existence, uniqueness and their geometric shape.  At the present moment, these aspects are relatively well understood in the case of the Mackey-Glass type diffusive equation (\ref{Mg}) and the KPP-Fisher nonlocal equation,  quite contrarily, they seem to be only sporadically investigated in the case of  the general ecological equation (\ref{17nl}). 

In the particular case of the food-limited model, we also give a rigorous  proof of  the existence of  a peculiar, yet  substantive and seemingly important class of non-monotone and non-oscillating wavefronts. As regards to the scalar models  coming from applications, this kind of wave solutions is analyzed here 
for the first time.

\subsection{On the  existence of semi-wavefronts}
\noindent By invoking  the approaches developed   in  \cite{HK,HTa}, in Sections 2 and 3 we establish  the  following general existence result: 

\begin{thm} \label{Te1} Assume that  the continuous function $G:\R_+\to \R$ satisfies  (\ref{Mo}) and (\ref{DG})  and that the non-negative kernel $K$  satisfies (\ref{ber}).  Then  for each speed $c \geq 2\sqrt{G^*}$ the equation (\ref{17nl})  has at least one semi-wavefront  $u(t,x) = \phi_c(x+ct)$.   Moreover,  (\ref{17nl})  does not possess  any semi-wavefront propagating with  the speed $c < 2\sqrt{G(0)}$.
\end{thm}
Theorem \ref{Te1} applies to the  above mentioned population models. In particular,  in the case of the food-limited model with spatiotemporal  interaction,  we have that 
$$
G(u) = \frac{1-u}{1+\gamma u}, \quad \gamma >0,  \ \mbox{so that } \ G(0)=G^*=1.
$$
Clearly, the Allee effect is not present here so that  the model has at least one semi-wavefront propagating with the speed $c$ if and only if 
$c \geq 2$.  

It is worth mentioning that the papers \cite{GG,GC,OW,TPT,WL,Pr} provide a series of 
conditions sufficient for the presence of {\em monotone } wavefronts in the food-limited equation.  On the other hand, as \cite{TPT} indicates and  we will also discuss  it later, it  is  rather unrealistic to expect  derivation of a sharp coefficient criterion for the existence of {monotone} wavefronts to this equation.  This explains the importance of the above simple criterion, $c \geq 2$,  for the existence of semi-wavefronts. The papers \cite{GG,GC,TPT} present computer simulations which confirm numerically  the validity of the above analytical result. Taking $\gamma=0$ in the food-limited equation, we obtain  the KPP-Fisher model with the {\it spatiotemporal}  interaction. In such a case, Theorem  \ref{Te1}  slightly  extends the existence criterion of \cite{BNPR,HK} proved for the KPP-Fisher model with the nonlocal {\it spatial}  interaction (i.e. when  $K(s,y) = K_1(y)\delta(s)$). 

Next,  by considering 
\begin{equation}\label{SAE}
G(u) = a +bu - cu^2, \quad \mbox{where} \ a, c > 0 \ \mbox{and} \ b \in \R, 
\end{equation}
we obtain  a  single species model analyzed in \cite{GL,HWF,SR,SPH}.  Since 
$$
G(0) = G^* = a \ \mbox{if and only if} \  b \leq 0 \quad  \mbox{and} \ G^* = a+ \frac{b^2}{4c}  > G(0)\ \  \mbox{if} \ b >0, 
$$
 this model possesses the Allee effect for  $b >0$.  Clearly,  the birth rate per capita $G(u)$ satisfies all the assumptions of Theorem \ref{Te1} and therefore 
equation (\ref{17nl}) with such $G(u)$ has at least one semi-wavefront propagating with the speed  $c \geq 2\sqrt{G^*}$ and does not have such a solution if $c < 2\sqrt{a}$. 

For some special kernels,   without admitting the Allee effect, the wavefronts to  (\ref{17nl}) with $G(u)$ given by   (\ref{SAE}) were investigated in \cite{HWF,SPH}. 
More  precisely, for special { spatiotemporal} averaging  kernels (\ref{ber}) (allowing the use of the so-called linear chain technique and containing some small parameter $\tau$, a delay)  and  $b <0$, the existence of wavefronts for  (\ref{17nl}),  (\ref{SAE})  was proved by  Song {\it et al} \cite{SPH}  with the help  of  Fenichel's invariant manifold theory.  Recently,  Han {\it et al} considered (\ref{17nl}),  (\ref{SAE})  with the singular kernels  $K(s,y) = K_1(y)\delta(s)$, $K_1(0) >0$, and also with $b < 0$, see  \cite{HWF}. Using the Leray-Schauder degree argument, they proved the existence of  semi-wavefronts for (\ref{SAE}) for each  propagation speed $c \geq 2\sqrt{a}$.   The above mentioned conclusions of \cite{HWF,SPH} follow from  our more general existence result. 

\subsection{On the semi-wavefront uniqueness}

The stability of semi-wavefronts implies their uniqueness up to translation \cite{LLLM} so that 
the uniqueness property of wave can be considered as a natural indicator of its stability.  This simple observation becomes important if we take into account significant  technical difficulties in proving the wave stability \cite{HJ, LLLM}. Now, we can observe a striking difference between nonlocal equations  (\ref{Mg}) and (\ref{17nl}) in what concerns  the uniqueness property of their wave solutions. Indeed, if the nonlinearity  $- u + F(v)$ in (\ref{Mg}) is sub-tangential at $0$ (i.e. $-u+F(v) \leq - u + F'(0)v,  \ u,v \geq 0$), the Diekmann-Kaper theory assures the uniqueness of each wave (including the critical one),  \cite{AGT}. However,   if the nonlinearity $u\;G(v)$ has the similar property 
(i.e. $u\;G(v) \leq u\;G(0),  \ u,v \geq 0$), equation (\ref{17nl})  for certain kernels $K$ can possess multiple wavefronts and  semi-wavefronts propagating with the same speed, \cite{HK}. Even so, a remarkable fact is that, in some special cases, monotone wavefront can still be unique in the class of {\it  all monotone wavefronts} \cite{FZ,MQG}. In particular, this is true when $G(u)$ is a linear function, $G(v)=1-v$. The latter result is due to  Fang and Zhao \cite{FZ} and it  can be generalized for general ecological equation  (\ref{17nl}) as follows: 
\begin{thm}\label{T12}  Assume that $G(u)$ is a strictly decreasing, Lipschtiz continuous  function which is differentiable at $1$ with $G'(1) <0$ and such that, for some $\alpha >0$, 
$$
G(u)/({u-1})-G'(1) = O((1-u)^\alpha), \ u \to 1^-. 
$$
 Furthermore, assume that for each  $c >0$ there are  $\lambda_0(c), \lambda_1(c) \in (0,+\infty]$ such that  
$$
\frak{I}_{c}(\lambda):= \int_{\R_+\times\R}  K(s,y)e^{-\lambda(cs+y)}dsdy \in \R, \ \  \frak{I}_{c}(0)=1, \ \lim_{\lambda\to- \lambda_0(c)^+}\frak{I}_{c}(\lambda)=+\infty, 
$$
for all $\lambda \in (-\lambda_0(c), \lambda_1(c))$ and
$\frak{I}_{c}(\lambda)$ is a scalar continuous function of variables  $c, \lambda$. 
Suppose that $\phi_c(t), \psi_c(t)$ are two monotone wavefronts to
 equation (\ref{twe2an}) propagating with  the same speed $c>0$.  Then there exists  $t'\in \R,$ such that $\phi_c(t)\equiv \psi_c(t+t')$.  \end{thm}
We note that Theorem \ref{T12} is a   non-trivial extension of the aforementioned uniqueness result from \cite{FZ}.  Indeed,  the proof in \cite{FZ} 
uses  in essential way  the sub-tangency property of the function $u(1-v)$  at the equilibrium $1$ (i.e. the inequality $u(1-v) \leq (1-v)$ for $u,v \in [0,1]$).  This property, however, is not  generally  satisfied by the function $u\,G(v)$ (e.g. in the case of the food-limited model with $\gamma >0$). 
Therefore, in order to prove Theorem \ref{T12}, it was necessary to find a completely different method.   
Our approach here is strongly motivated by the recent studies in \cite{HT,ST}.  The proof of Theorem \ref{T12} is given in Section 4.

\subsection{On the existence of non-monotone and non-oscillating wavefronts}

It is well known  that the classical  (i.e. non-delayed and without non-local interaction) scalar  monostable diffusive equation cannot admit waves other than wavefronts. Moreover, these  wavefronts must have monotone profiles. 
This simple panorama changes drastically if the equation incorporates either delayed or non-local interaction  effects. In such a model, non-monotone waves with unusually  high leading edge can appear: clearly, this type of waves might produce a major impact in  the underlying  biological system \cite{BY2,GL,HWF}.  Therefore, as it was mentioned in \cite{SWZ}, {\it `it is important and challenging, both theoretically and numerically, to find this critical value} [when the wave monotonicity is lost] {\it and to understand the mechanism behind this loss of monotonicity of wavefronts'.} 
  
Ashwin {\em et al}, \cite{ZAMP}, was the first author who provided numerical evidence suggesting a  clear relation between the shape of the wave profile and the position on the complex plane  of eigenvalues to the  profile equation (\ref{twe2an}) linearized around the positive steady state.  The heuristic  paradigma suggested by \cite{ZAMP}  can be considered as  a particular case of the so-called {\it linear determinacy principle} \cite{LLW} and  reads as follows: {\it  If the linearization around the positive equilibrium $\kappa$ has a negative eigenvalue, the  wave profile is monotone; next, if this  linearization  does not have  negative eigenvalues and also does not have complex eigenvalues with the positive real part,  then the wave profile {oscillatory converges}  to $\kappa$; finally,  if this  linearization  does not have a negative eigenvalue but does have complex eigenvalues with the positive real part,  then the wave profile develops non-decaying oscillations   around $\kappa$. } In \cite{ZAMP,GL}, the authors tested this principle on the  non-local and delayed KPP-Fisher equations; numerical simulations  realized in subsequent works  also supported  the above informal  principle  for the food-limited model \cite{GG,GC,OW}, the model with the quadratic function $G(u)$ given by (\ref{SAE}) \cite{HWF} and  the Mackey-Glass type diffusive equations \cite{BY2,LW,LLLM}. Hence, a preliminary answer  to the above  concern in \cite{SWZ} is based on abundant numerical evidence and might be formulated as follows: {\it The wavefront profile loses its monotonicity and starts to oscillate at $+\infty$ around the positive equilibrium $\kappa$ at the moment when the negative eigenvalues of the  linearization of the profile equation at $\kappa$ coalesce  and then disappear}. 
\begin{figure}[h] \label{FF2}
\centering \fbox{\includegraphics[width=9cm]{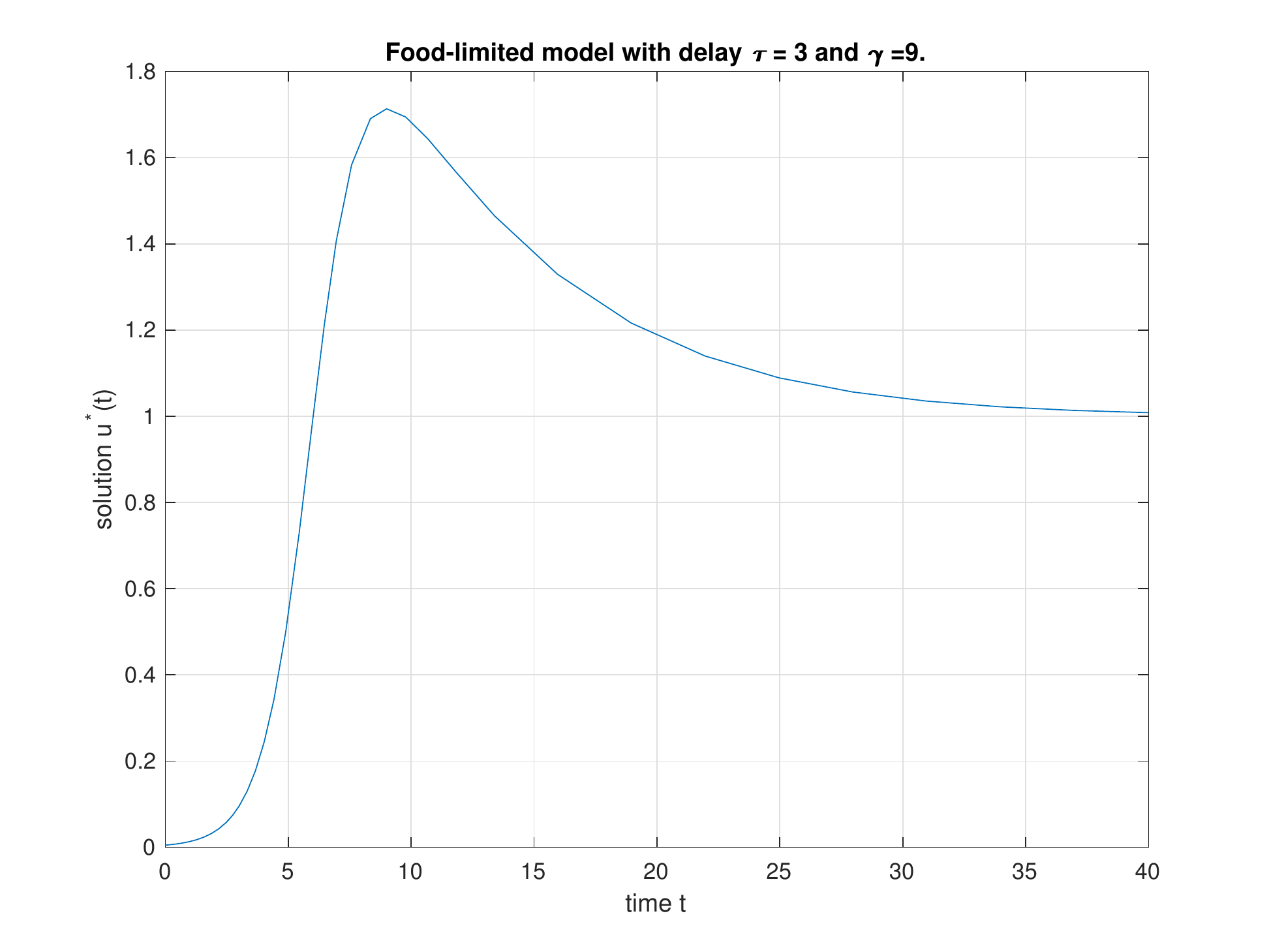}}
\caption{\hspace{0cm} Non-monotone non-oscillating wavefront for equation (\ref{i1dIN}) in the case when $\tau=3, \ \gamma = 9$ and $c\gg 2$.} 
\end{figure}
The above answer is very significant from the practical point of view. Indeed, it allows to indicate the linearly determined `safe' zone of the model parameters where we cannot 
expect appearance of an invasion traveling wave with dramatically high concentration of acting agents in its leading edge (what we can observe  on Figure 1, where $\kappa=1$).   The circumstance that this non-monotone wave is additionally developing oscillations at its rear part seems to be less important: actually, sometimes the oscillatory component is decaying so fast that the oscillating (at $+\infty$) wave 
can visually be interpreted as (or approximated by) eventually monotone wave, e.g. see \cite{GL}.

Importantly, the above mentioned monotonicity criterion can be analytically  justified for some subclasses of equations (\ref{Mg}) 
and (\ref{17nl}) including the KPP-Fisher nonlocal \cite{FZ} and the delayed \cite{ADN,GT,GTLMS,HTa,HTb,ST} equations, some particular cases  of  the Mackey-Glass type delayed  \cite{GTLMS} and non-local \cite{MQG} equations. Analyzing related proofs, we can see in each of them that the reaction term is necessarily dominated by its linear part at the positive steady state.  The key  discovery of the present subsection is that without assuming  this sub-tangency condition on the nonlinearity  $u\,G(v)$ at the equilibrium $1$, i.e. without requiring the inequality 
$$
G(v) \leq G'(1)(v-1), \  v \in [0,1], 
$$
the ecological equation (\ref{17nl}) might not satisfy the above heuristic principle. The mechanism behind the unexpected loss of monotonicity of wavefronts in such a case is precisely the same one which causes the  {\it "linear determinacy principle"} \cite{LLW} to fail   for the model exhibiting the Allee effect (which  finally results in the appearance of  pushed waves).  Specifically, we will show that  the food-limited model with spatiotemporal interaction  admits unexpectedly high wavefronts for a broad domain of parameters $\tau, \gamma$  from an apparently `safe'  zone provided by the linear analysis of the model at the positive steady state, see Figure 1 and the  next two Theorems. 
\begin{thm} \label{T20IN} For each fixed $\tau >0, \gamma >0$ and $c\geq 2$ the food-limited equation 
\begin{equation}\label{i1IN}
\partial_tu(t,x)=\partial_{xx}u(t,x) +u(t,x)\left(\frac{1-(K *u)(t,x)}{1+\gamma (K*u)(t,x)}\right), \quad x \in \mathbb R,
\end{equation}
with the  so-called weak generic delay kernel 
\begin{equation}\label{kernelIN}
K(s,y,\tau)= \frac{e^{-y^2/4s}}{\sqrt{4\pi s}}\frac 1\tau e^{-s/\tau}
\end{equation}
has at least one positive wavefront $u(t,x)=\phi_c(x+ct)$. The profile $\phi_c(t)$ either tends to $1$ as $t \to +\infty$ or is asymptotically periodic at $+\infty$. 
If, in addition,  $\tau >(1+\gamma)/4$ then $\phi_c(t)$ is oscillating around  $1$ on some interval $[A,+\infty)$.  Furthermore, if $\tau \leq (1+\gamma)/4$   then 
there exists   $\hat c(\tau,\gamma)\geq 2$ such that  $\phi_c(t)$  is eventually monotone at $+\infty$ whenever $c \geq \hat c(\tau,\gamma)$.  Finally, for each  $\gamma > 7.29\dots$  there are positive $c_0(\gamma), \tau_\#(\gamma) < (1+\gamma)/4$ such that for each 
$\tau \in (\tau_\#(\gamma), (1+\gamma)/4)$ and $c \geq c_0(\gamma)$ there exists  a wavefront whose profile $\phi_c(t)$  is neither monotone nor oscillating. 
\end{thm}
 Hence, the aforementioned heuristic criterion fails for $\tau \in (\tau_\#(\gamma), (1+\gamma)/4)$ if the propagation speed is sufficiently large.  Figure 2 below presents the corresponding  region of parameters  (lying between the graphs of  $\tau = (1+\gamma)/4$ and $\tau = \tau_\#(\gamma), \ \gamma \in [10,40]$). 
\begin{figure}[h]
\centering \fbox{\includegraphics[width=9cm]{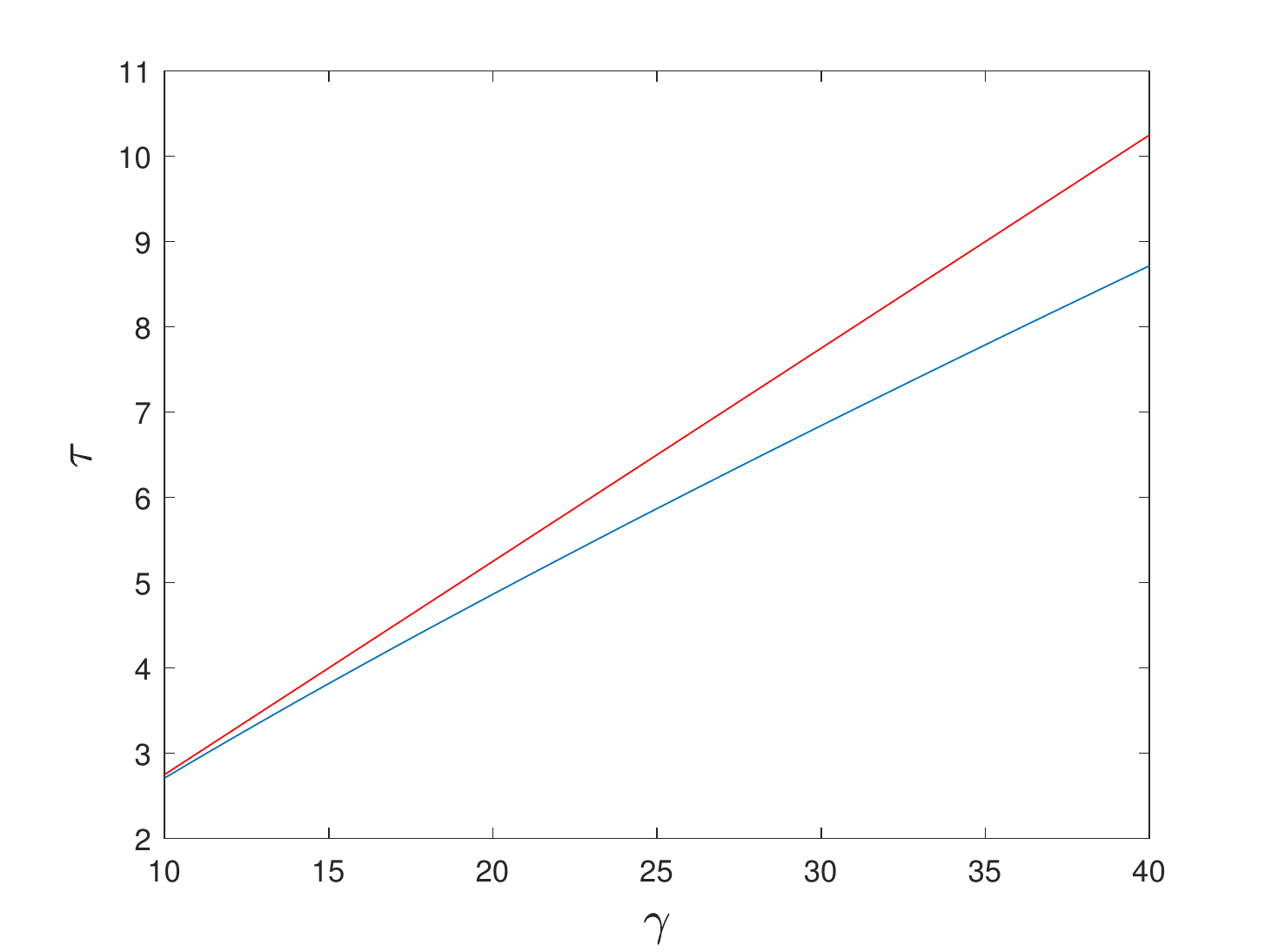}}
\caption{\hspace{0cm}   For each pair $(\gamma, \tau)$ of parameters lying in the domain between the graphs of  $\tau = (1+\gamma)/4$ and $\tau = \tau_\#(\gamma)$ equation (\ref{i1IN}) possesses fast  non-monotone and non-oscillating wavefronts.}
\end{figure}
It should be noted that  the existence of monotone wavefronts for (\ref{i1IN}) was recently proved in \cite{TPT} under the condition
 \begin{equation*}\label{ex1c}
 \tau \leq \left\{
\begin{array}{ccc}
(1+\gamma)/4 & ,& \gamma \in (0,1],\\
\gamma/(1+\gamma) & , & \gamma \geq 1.  
\end{array} \right. 
\end{equation*}
Importantly, for values of  $\gamma \in (0,1]$, the inequality $\tau \leq (1+\gamma)/4$  gives a sharp criterion for the existence of monotone wavefronts. Consequently, in such a case, the above mentioned heuristic criterion holds. 
A question left unanswered in \cite{TPT} concerns the presence of non-monotone and non-oscillating wavefronts (and, more generally, semi-wavefronts) for the values  $\gamma > 1$ and $\tau > \gamma/(1+\gamma)$. In this context, Theorem \ref{T20IN} explains phenomenon numerically observed in \cite[Figure 1]{TPT} for the values $\gamma=40, \tau =9$. 

Theorem \ref{T20IN} (as well as  Theorem \ref{T30IN}  below) will be proved in Section 5 with the help of a) Mallet-Paret and H. Smith theory of monotone cyclic feedback systems \cite{El, MPSm, MPSe} and b) the singular perturbation theory developed by Faria {\it et al} in \cite{fhw,FTnl}. The latter theory 
provides a rigorous justification of the Canosa method \cite{Can,GL,Mur,OW} for the case of equations incorporating spatiotemporal effects.  In \cite{Can}, Canosa constructed an analytic approximation of the monotone wavefront for the classical KPP-Fisher equation, which  is   highly  accurate for all values of $c \geq 2$, although theoretically valid only for small  $\epsilon:=  c^{-2}  \ll 1$ .  This allowed Murray to  observe in \cite{Mur} that {\it `It is an encouraging fact that asymptotic solutions with `small' parameters ... frequently give remarkably accurate solutions'}. Now, it is also known that in models with spatiotemporal effects,  the wavefronts propagating with smaller speeds generally have better monotonicity and convergence properties
than the wavefronts propagating with  bigger speeds, cf. \cite{HTa,HTb}. So, taking into account all  these arguments  and numerical simulations in \cite{GC}, we conjecture that the food-limited model (\ref{i1IN}) with the kernel (\ref{kernelIN}) cannot have proper semi-wavefronts (i.e. we conjecture that always  $\phi_c(+\infty) =1$). It is also clear from Theorem \ref{T20IN} that, in difference with some Mackey-Glass type equations \cite{LLLM},  model (\ref{i1IN})  cannot have waves whose  profiles oscillate `chaotically' around 1 at $+\infty$. 

\vspace{2mm}

Similarly,  we can establish the existence of non-monotone non-oscillating wavefronts in the linearly 
determined domain $0< \tau < (1+\gamma)/e$ of parameters $(\gamma,\tau)$  
for the food-limited model with single discrete delay
\begin{equation}\label{i1dIN}
\partial_tu(t,x)=\partial_{xx}u(t,x) +u(t,x)\left(\frac{1-u(t-\tau,x)}{1+\gamma u(t-\tau,x)}\right), \quad x \in \mathbb R.
\end{equation}
To give more complete description of the possible shapes of wavefronts,  we recall  the definition of
sine-like slowly oscillating profile \cite{HTa,mps,mps2,TTT}:
\begin{definition} Set $h:= c\tau,\  \mathbb{I} = [-h,0] \cup \{1\}$. For each
$v \in C(\mathbb{I})\setminus\{0\}$ we define the number of sign
changes by $$\hspace{-1mm} {\rm sc}(v) = \sup\{k \geq 1:{\rm  there
\ are \ } t_0 <
 \dots < t_k  \ {\rm such \ that\ }
v(t_{i-1})v(t_{i}) <0 {\rm \ for \ }  i\geq 1\}. $$ We set ${\rm
sc}(v) =0$ if $v(s) \geq 0$ or  $v(s) \leq 0$ for $s \in
\mathbb{I}$. If $\phi$ is a non-monotone semi-wavefront profile  to (\ref{i1dIN}), we set \ $(\varphi_t)(s) =
\phi(t+s)- 1$ if $s \in [-h,0]$, and $(\varphi_t)(1) =
\phi'(t)$. We will say that $\phi(t)$  is sine-like slowly oscillating on a connected interval $\mathfrak{I}$
 if the following conditions are satisfied: (d1) $\phi$ oscillates around $1$ and has exactly one critical point between each two
 consecutive intersections with  level $1$; (d2) for each $t \in \mathfrak{I}$, it holds  that either sc$(
\varphi_t)=1$ or sc$(\varphi_t)=2$.
\end{definition}

Note that  if $\phi$ sine-like slowly oscillates on some interval $\mathfrak{I}$ and  if $\{Q_j\}_{j \geq 1},$ $ Q_j \in \mathfrak{I}$ denotes
the increasing sequence of all moments $Q_j$ where  $\varphi(Q_j)=1$, then  $Q_{j+2} - Q_j \geq h$ for all $j \geq 1$. Thus 
every open time interval  of  length $h$  can contain at most two points at which the graph
of $\varphi = \varphi(t)$  crosses level 1. 

Observe also that the uniqueness conclusion  in the next theorem is much stronger than in Theorem \ref{T12}. 
\begin{thm} \label{T30IN} For each fixed triple of parameters $c \geq 2,\ \tau \geq 0, \ \gamma \geq 0$, equation (\ref{i1dIN}) has a unique (up to translation) positive semi-wavefront 
$u(t,x) = \phi_c(x +ct)$. The profile $\phi_c(t)$ is either eventually monotone or is sine-like slowly oscillating around $1$ at $+\infty$.   Next, 
for each $0< \tau < (1+\gamma)/e$ such that 
$$
\zeta := \max_{a\in [0,1]} a \exp\left(\tau + \frac{(1-a)\tau}{1+\gamma a}\right)\left(\frac{1+a\gamma}{1+a\gamma e^{\tau}}\right)^{1+1/\gamma} >1, 
$$
there exists  $\hat c(\tau,\gamma)\geq 2$ such that for each $c \geq \hat c(\tau,\gamma)$ equation (\ref{i1dIN}) has 
a  positive wavefront propagating with the speed $c$ and whose profile $\phi_c(t)$ is eventually monotone at $\pm \infty$ and is non-monotone on $\R$. In fact, $|\phi_c(\cdot)|_\infty \geq \zeta >1$. 
\end{thm}
\begin{figure}[h] \label{FF3}
\centering \fbox{\includegraphics[width=9cm]{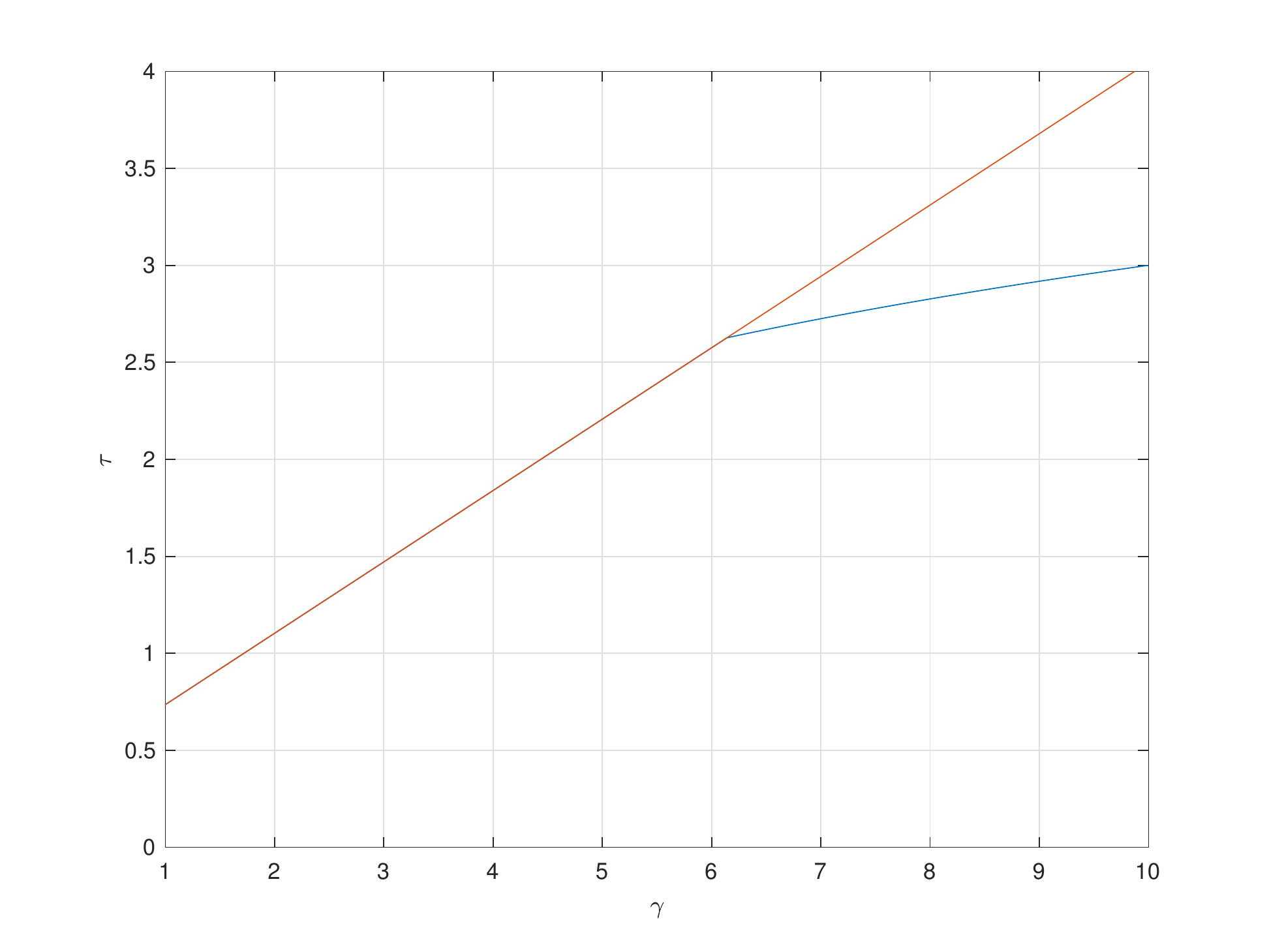}}
\caption{\hspace{0cm} Triangular domain between two curves in the top-right corner corresponds to the parameters $\gamma, \tau$ satisfying assumptions of Theorem \ref{T30IN}.}
\end{figure}
On Figure 3, we present the subset of parameters $(\gamma,\tau)$ lying in the rectangle  $[1,10]\times[0,4]$ and  satisfying the assumptions of Theorem \ref{T30IN}.

In view of  Ducrot and Nadin work \cite{ADN} on (\ref{i1dIN}) with $\gamma =0$ and  Mallet-Paret and Sell theory in \cite{mps}, we  conjecture that, in full analogy with the statement of Theorem \ref{T20IN},  the wave profiles $\phi_c(t)$ provided by Theorem \ref{T30IN} 
cannot oscillate `chaotically' around 1 at $+\infty$ and should   
either converge to $1$ as $t\to +\infty$, or  approach a non-trivial periodic regime  at $+\infty$.

As far as we know, the food-limited equations (\ref{i1IN}) and  (\ref{i1dIN}) are the first scalar models  coming from applications where untypical behavior due to the presence of non-monotone non-oscillating wavefronts is established analytically.  Between previous studies, we would like to mention numerical simulations in \cite{TPT} and the theory developed  in \cite{IGT}  for the "toy" example  of the Mackey-Glass type diffusive equation with a single delay.  In view of the argumentation exposed in \cite[Subsection 2.3]{GTLMS} and also   in this work, we conjecture  that the celebrated Nicholson's diffusive equation with a  discrete delay (i.e.    taking  ${F((K*u )(t,x))}: = pu(t-\tau,x)e^{u(t-\tau,x)}$ in  (\ref{Mg}))  possesses non-monotone non-oscillating wavefronts when  $p > e^2$ and $\tau$ is bigger than certain critical value.

\section{Some auxiliary results } \label{ETW}
Following \cite{HK,HTa}, we consider equation (\ref{twe2an}) together with   
\begin{eqnarray} \label{twe2m} &&
\phi''(t) - c\phi'(t) + g_\beta(\phi(t)){G((N_c*\phi)(t))} =0,  
\end{eqnarray}
where  the continuous piece-wise linear function $g_\beta, \ \beta >1,$ is given by 
\begin{eqnarray}\label{G}
g_{\beta}(u)=\left\{\begin{array}{cc} u,& u \in [0,\beta], \\    \max\{0,2\beta -u\},
& u>  \beta.\end{array}\right. 
\end{eqnarray}
Observe that equation (\ref{twe2m}) has two isolated equilibria $\phi(t) \equiv 0$ and $\phi(t) \equiv 1$ as well as an interval $[2\beta, +\infty)$ of constant solutions.  
We have the following 
\begin{lem} \label{beA} Assume that $\phi, \ \phi(-\infty) =0,$ is a non-negative, bounded and non-constant solution to
(\ref{twe2m}). Then $\phi(t) \leq 2\beta$ for all $t \in \R$.  Next, if  either $t_0$ is a point of local maximum for $\phi(t)$ with $\phi(t_0) < 2\beta$ or  $t_0$ is the smallest number such that $\phi(t_0) =2\beta$, then $(N_c*\phi)(t_0) \leq 1$.  
\end{lem}
\begin{proof} 
On the contrary, suppose that there exists a maximal interval  $(t_0, t_1)$, such that   $\phi(t) > 2\beta = \phi(t_0)$ for all $t \in (t_0, t_1)$. Then  $\phi'(t_*) > 0, \phi(t_*) > 2\beta$ for some $t_* \in (t_0,t_1)$. { It follows from (\ref{twe2m}) and the definition of $g_\beta$ that }
$\phi''(t) = c\phi'(t)$ for all  $\ t \in (t_0,t_1)$.  Hence, $\phi'(t) = \phi'(t_*)e^{c(t-t_*)} >0,$ $t \in (t_0,t_1)$ and therefore $t_1= +\infty$, $\phi(+\infty) = +\infty$,  contradicting the boundedness of  $\phi$.

If   $t_0$ is a point of local maximum of $\phi(t)$,  then $\phi'(t_0) = 0,$
 $\phi''(t_0) \leq  0$. If, in addition, $\phi(t_0) <2\beta$ then $g_\beta(\phi(t_0)) >0$ and  thus (\ref{Mo}) and (\ref{twe2m}) 
assures that  $(N_c*\phi)(t_0) \leq 1$. Now, if $t_0$ is the smallest number such that $\phi(t_0) =2\beta$,
 then clearly there exists a sequence 
 $t_j \to t_0,$ $ t_j < t_0, \ j \in \N,$ such that $\phi'(t_j) >0, \ \phi''(t_j) <0, \ \phi(t_j) < 2\beta$. But then 
$(N_c*\phi)(t_j) < 1,$ for all $j$ and therefore also $(N_c*\phi)(t_0) \leq 1$. \qed
\end{proof}
As it is usual for the monostable systems,  solutions to  equations  (\ref{twe2an}) and (\ref{twe2m}) exhibit   
the following separation dichotomy at $\pm\infty$:
\begin{lem} \label{bebe} Assume that $\phi$ is a non-negative, bounded and non-constant solution of 
(\ref{twe2m}) or (\ref{twe2an}). Then $\phi(t) >0, \ t \in \R$. If, in addition, $\phi(t_n) \to 0$ along some sequence $t_n \to -\infty,$ then there exists $\rho$ such that $\phi(t)$ is  increasing on some interval $(-\infty, \rho], \ \phi(-\infty) = 0,$  
$
\liminf_{t \to +\infty} \phi(t) >0, 
$
and $c \geq 2\sqrt{G(0)}$. 
\end{lem}
\begin{proof}   Since 
equation (\ref{twe2m})  with $\beta=+\infty$ coincides with  (\ref{twe2an}),  it suffices to consider  equation (\ref{twe2m}) allowing $\beta=+\infty$. 

First, notice that $y=\phi(t)$ is the solution of the following initial value problem 
for a linear second order ordinary differential equation
\begin{equation}\label{hi}
y''(t) -cy'(t) +a(t)y(t)=0, 
\end{equation}
where 
$$
a(t) := \left\{\begin{array}{cc} G((N_c*\phi)(t)),& 0 \leq\phi(t) \leq \beta, \\  
\frac{g_{\beta}(\phi(t))}{\phi(t)} G((N_c*\phi)(t)), & \phi(t)>  \beta,
\end{array}\right.
$$
is a continuous bounded function. Suppose for a moment that $\phi(s)= 0$. Since $\phi(t) \geq 0,$ $t \in \R,$ this 
yields $\phi'(s)=0$. 
But then 
$y(t)  \equiv 0$ due to the uniqueness theorem, a contradiction. Therefore $\phi(t)$ is positive for all $t \in \R$. 
 
In the sequel,  we follow closely the argumentation presented in  \cite[Lemmas 3.7 and 3.9]{BNPR}.
Let us assume that the {second} conclusion of the lemma is false. As $\phi(t_n) \rightarrow 0$ for {$t_n \rightarrow - \infty$} and $\phi(t)$ is
not eventually monotone at $-\infty$, there exists another sequence $s_j \rightarrow  - \infty$ such that $\phi(t)$ attains a local
minimum at $s_j$ and $\phi(s_j) \rightarrow 0$.
Since $a(t)$ is  a continuous bounded function and  $\phi(t)$ is bounded in $C^2(\R)$, 
we can apply the  Harnack inequality, see  \cite[Theorem 8.20]{GiTr},  to equation (\ref{hi}). We can  conclude that for any
$R > 0$ and any $\delta >0$ there exists  $n_0 \in \N$ such that  $0< \phi(t) \leq \delta$ for all $t  \in (s_j -R, s_j + R)$ and $j\geq n_0$. 
In particular,   $g_\beta(\phi(s_j)) = \phi(s_j) >0, \ j \geq n_0$, so that  it follows from (\ref{Mo}) and 
(\ref{twe2m})   that $
 G((N_c*\phi)(s_j)) \leq 0, $  implying $(N_c*\phi)(s_j) \geq 1$. On the other hand, if we take  $\delta < 1/2$ and $R$ sufficiently large to have 
 $\int_{-R}^RN_c(s)ds > 1 - 1/(4\beta)$ and recalling that, by Lemma \ref{beA}, it holds that  $\phi(t) \leq 2\beta$, we obtain the following contradiction:
 $$
 (N_c*\phi   )(s_j) \leq \int_{-R}^R\phi(s_j-s)N_c(s)ds + 2\beta\int_{|s| \geq R}N_c(s)ds < 1/2 +1/2=1. 
 $$
To prove the third conclusion of the lemma, 
let us assume that $
\liminf_{t \to +\infty} \phi(t) =0
$. Then  there exists a sequence $t_k \to   +\infty$ such that $\phi(t_k) \to  0$ so that, following the above reasoning, we may conclude that 
$\phi(+\infty)=0$ and $\phi'(t)<0$ on some  interval $[\varrho, + \infty)$.  Consequently, $a(t) = G(0)+o(1)$ as $t\to \pm \infty$. This shows that 
$c\not=0$, since otherwise $\phi''(t) = - a(t)\phi(t) <0, \ \phi'(t) <0$ for all large positive $t$ implying that $\phi(+\infty) = -\infty$.  On the other hand, if $c>0$ [respectively, $c<0$] then 
equation (\ref{hi}) is exponentially unstable [respectively, stable] at both $+\infty$ and $-\infty$.  This means that $\phi(t)$ can vanish only at one end of the real line, 
being separated from zero at the opposite end of $\R$. Actually, since by our assumption $\phi(-\infty)=0$, this  implies that equation (\ref{hi})  is unstable 
and therefore $c >0$ and $
\liminf_{t \to +\infty} \phi(t) >0. 
$
Furthermore, due to  a classical oscillation theorem by Sturm (e.g. see \cite{JW}), the solution $\phi(t)$ oscillates around $0$ at $-\infty$ once  $a(-\infty) = G(0) > c^2/4$. 
\qed
\end{proof}
In fact, for a  positive solution $\phi(t)$ converging at $+\infty$, the limit value $\phi(+\infty)$ is either $1$ or $2\beta$:
\begin{lem} \label{li} Let a positive   non-constant $\phi$  solve (\ref{twe2m}), $\phi(-\infty)=0$ 
and there exists finite  limit 
$\phi(+\infty)$. Then $\phi(+\infty) \in \{1,2 \beta\}.$    If  $\phi(+\infty) =2\beta$ then $\phi(t) \equiv 2\beta$ on some maximal nonempty interval $[T_1, + \infty)$ and $(N_c*\phi)(T_1) \leq 1$. Furthermore, if $ {2}\beta \int_{-\infty}^0N_c(s)ds >1$ then 
 $\phi(+\infty) =1$.
\end{lem}
\begin{proof} It follows from Lemmas \ref{beA} and \ref{bebe} that $\phi(+\infty) \in (0, 2\beta]$. In addition, 
if $\phi(+\infty) \not\in \{{1}, 2\beta\},$ then for
$$
r_*(t): = g_\beta (\phi(t)){G((N_c*\phi)(t))},   
$$
we have that 
$$
\lim_{t \to +\infty}r_*(t)   =  g_\beta (\phi(+\infty)){G((\phi(+\infty))}   \not= 0.
$$

However,  in this case  the differential equation 
$\phi''(t)- c\phi'(t) + r_*(t) =0$ does not have any 
convergent bounded solution on $\R_+$. Indeed,  we have that  
$$
|\phi'(t)| = \left|\phi'(s) + c(\phi(t)-\phi(s)) -\int_s^tr_*(u)du\right| \to +\infty \ \mbox{as} \ t \to +\infty. 
$$
Finally, assume that $\phi(+\infty) = 2\beta$, then  there exists $T_1 \in \R$ such that  $r_*(t)\leq 0$ for $t\in [T_1, \infty)$  and thus
$\phi''(t) -c\phi'(t) \geq 0$  for $t \geq T_1$. As a consequence, $\phi'(t) \geq \phi'(s)e^{c(t-s)}$ for $t\geq s \geq T_1$. If $\phi'(s) >0$ for some $s \geq T_1$, we obtain a contradiction:
$\phi'(+\infty) = +\infty$. Therefore we have to analyse the case when $\phi'(s) = 0$ for all $s \geq T_1$
(we can assume that $T_1$ is the smallest number with such a property).  By Lemma \ref{beA}, 
$${2}\beta \int_{-\infty}^0N_c(s)ds =\int_{-\infty}^0\phi(T_1-s)N_c(s)ds \leq (N_c*\phi)(T_1) \leq 1,$$ 
which proves the last statement of the lemma. \qed
\end{proof}
\begin{remark}\label{ir}
Suppose that  $\int_{-\infty}^0N_c(s)ds > 0$. Then we can choose $\beta$ large enough to meet the inequality  $ {2}\beta \int_{-\infty}^0N_c(s)ds >1$. Hence, if $\int_{-\infty}^0N_c(s)ds > 0$ and $\phi(+\infty)$ exists,   we can assume that   $\phi(+\infty) =1$.
\end{remark}
In what follows, we assume that $c \geq 2\sqrt{G(0)}$ and 
\begin{equation}\label{E}
\lambda(c): = \frac{1}{2}(c- \sqrt{c^2-4G(0)})  \leq \mu(c): =  \frac{1}{2}(c+\sqrt{c^2-4G(0)}) 
\end{equation}
will denote the  roots of the characteristic  equation $z^2 -cz + G(0) = 0$ at the zero equilibrium,  note that they both are positive. 

Before proving the next lemma,  we observe that the second differentiability assumption of (\ref{DG}) implies the existence of $L>0$ such that 
\begin{equation}\label{LL}
G(u) \leq L(1-u), \quad u \in [0,1]. 
\end{equation}
\begin{lem} \label{po} Let a non-negative bounded  $\phi \not\equiv 0$ solve either (\ref{twe2m}) or (\ref{twe2an}) and $c \geq 2\sqrt{G(0)}$. Assume also 
that $G(s) <G(0)=G^*$ for all $s >0$. 
Then 
$$- \phi'(t)/ \phi(t) > -\lambda(c)
.$$  
If, in addition,  $\phi(-\infty)=0,$
$\phi(t) \leq 1, \ t \in \R,$  then  $\phi'(t) >0$ for all $t \in \R$ and  $\phi(+\infty) = 1$. \\
\end{lem}
\begin{proof}Again, it suffices to consider  equation (\ref{twe2m}) allowing $\beta=+\infty$. 
Suppose now that $\phi$  satisfies (\ref{twe2m}) and $c> 2\sqrt{G(0)}$. 
Set 
$${\mathcal N}(\phi)(t):= G(0)\phi(t)- g_\beta (\phi(t))\,G((N_c*\phi)(t)),
$$ 
then ${\mathcal N}(\phi)(t)> 0$ and 
\begin{equation}\label{re}
\phi(t) = \frac{1}{\mu-\lambda}
\int_t^{+\infty}(e^{\lambda (t-s)}- e^{\mu
(t-s)}){\mathcal N}(\phi)(s)ds,
\end{equation}
where $\lambda :=\lambda(c) < \mu:=\mu(c)$ {are given by (\ref{E}).} 
As a consequence, we have that 
$$
\phi'(t) = \frac{1}{\mu-\lambda}
\int_t^{+\infty}(\lambda e^{\lambda (t-s)}-\mu  e^{\mu
(t-s)}){\mathcal N}(\phi)(s)ds,
$$
and therefore
\begin{equation}\label{bl}
\phi'(t) -\lambda \phi(t) = - \int_t^{+\infty}e^{\mu
(t-s)}{\mathcal N}(\phi)(s)ds < 0. 
\end{equation}
If now $c=2\sqrt{G(0)}$ (i.e. $\lambda = \sqrt{G(0)}$), we find similarly that 
\begin{equation}\label{re1}
\phi(t) =
\int_t^{+\infty}(s-t)e^{\sqrt{G(0)}(t-s)}{\mathcal N}(\phi)(s)ds,
\end{equation}
$$
\quad \phi'(t) =
\int_t^{+\infty}[\sqrt{G(0)}(s-t)-1]e^{\sqrt{G(0)}(t-s)}{\mathcal N}(\phi)(s)ds,
$$
and thus also 
$$
\phi'(t) -\sqrt{G(0)}\phi(t)=  -\int_t^{+\infty}e^{
\sqrt{G(0)}(t-s)}{\mathcal N}(\phi)(s)ds < 0. 
$$
Finally, $0 < \phi(t) \leq 1, \ t \in \R,$ implies  that $\phi''(t) - c\phi'(t) \leq 0$.  As a consequence, $ \phi'(t) \leq \phi'(s)e^{c(t-s)}, \  t \geq s$, so that if $\phi'(s) <0$ at some point $s$ then $\phi''(t) \leq c\phi'(t) <0 $ for all $t >s$. 
This implies $\phi(+\infty)=-\infty$, a contradiction.  Thus $\phi'(t) \geq 0$ for all $t \in \R$ so that, in view of Lemma  
\ref{li}, we obtain $\phi(+\infty)=1$. Next, if $\phi'(s)=0$ at some $s$ then $0 \leq \phi'(t) \leq \phi'(s)e^{c(t-s)}=0$ for all $t >s$ that yields $\phi(t)=1, \ t\geq s$. 

Consequently,  there exists  the leftmost $T_1 \in \R$ such that 
$\phi'(t) =0$ for all $t \geq T_1$ and $\phi'(t) > 0$ for $t < T_1$. But then $( N_c*\phi) (T_1)=1$, which implies $N_c(s) = 0$ a.e. on $\R_+$. Now, observe that both $\phi(t)$ and $1$ satisfy equations (\ref{re}), (\ref{re1}) and that 
${\mathcal N}(\phi)(t)={\mathcal N}(1)(t)=G(0)$ for $t \geq T_1$ and ${\mathcal N}(\phi)(t)= \phi(t)\left[G(0)-G(N_c*\phi)(t))\right]$ for $t \leq T_1$. Let $L>0$ be as in (\ref{LL}).  Then (\ref{re}) implies that, for $t<T_1$ close to $T_1$, $c>2\sqrt{G(0)}$,
$$
0< 1 -\phi(t) = \frac{1}{\mu-\lambda}
\int_t^{T_1}(e^{\lambda (t-s)}- e^{\mu(t-s)})(G(0)-\phi(s)[G(0)-G((N_c*\phi)(s))])ds
$$
$$
\leq \frac{1}{\mu-\lambda}
\int_t^{T_1}(e^{\lambda (t-s)}- e^{\mu(t-s)})(G(0)-\phi(s)[G(0)- L\, (1 - (N_c*\phi)(s))])ds
$$
$$
\leq \frac{1}{\mu-\lambda}
\int_t^{T_1}(e^{\lambda(t-s)}- e^{\mu(t-s)})ds\, (G(0)-\phi(t)[G(0)- L\,(1-\phi(t))])
$$
$$
= \frac{1}{\mu-\lambda}
\int_t^{T_1}(e^{\lambda (t-s)}- e^{\mu(t-s)})ds\, (1-\phi(t))[G(0)+ L\, \phi(t)] 
$$
$$
= (t-T_1)^2(0.5+o(1))(1-\phi(t))[G(0)+L\, \phi(t)], \ t \to T_1^{-}, 
$$
a contradiction. \\
Similarly, for  $c=2\sqrt{G(0)}$, (\ref{re1}) implies that, for $t<T_1$ close to $T_1$,
$$
0< 1-\phi(t) = 
\int_t^{T_1}(s-t)e^{\sqrt{G(0)}(t-s)}(G(0)-\phi(s)[G(0)-G((N_c*\phi)(s))])ds
$$
$$
\leq 
\int_t^{T_1}(s-t)ds\,(1-\phi(t))[G(0) + L\, \phi(t)] 
= 
\frac{(t-T_1)^2}{2}(1-\phi(t))[G(0)+ L\, \phi(t)], \ t \to T_1^{-}, 
$$
a contradiction. This finishes the proof of the lemma.
\qed
\end{proof}
\begin{lem} \label{ogran}  Assume 
that $G(s) <G(0)=G^*$ for all $s >0$. 
Then for each  $c\geq 2\sqrt{G(0)}$ and $N \in \mathcal{S}:= \{Q \in L^1(\R,\R_+): |Q|_1=1\}$ there exists  $U(c,N) \geq 1$ depending only on $c$ and $N$ such that the following holds: if $\phi(t)$, $\phi(-\infty) =0$,  is a positive bounded  solution of the  equation 
\begin{eqnarray} \label{twe2M} &&
\phi''(t) - c\phi'(t) + g_\beta(\phi(t)){G((N*\phi)(t))} =0,  
\end{eqnarray}
with $\beta > U(c,N)$,  then 
\begin{equation}\label{iu}
0< \phi(t) \leq U(c,N), \ t \in \R
\end{equation}
(i.e. the set of all semi-wavefronts to (\ref{twe2m})  is uniformly bounded by a constant which does not depend on a particular semi-wavefront). Moreover, given a fixed pair $(c_0,N_0) \in [2\sqrt{G(0)},+\infty) \times \mathcal{S}$,  we can assume that  the map 
$U: [2\sqrt{G(0)},+\infty) \times \mathcal{S}  \to (0,+\infty)$ is locally continuous at $(c_0,N_0)$. 
\end{lem}
\begin{proof} Similarly to the previous lemma, here we follow closely \cite{HK}. 
First, we take $U(c,N) \geq 1$ defined by one of the following mutually non-exclusive formulas: 
\begin{itemize}
\item  if $\int_{0}^{+\infty}N(s)ds >0$, then $U(c,N)= \left(\int_{0}^{+\infty}e^{-\lambda(c)s}N(s)ds\right)^{-1}$;
\item if $\int_{0}^{+\infty}N(s)ds < 0.001$, then $U(c,N)= 2\exp{(\lambda(c) (r+\sigma))},$ where $r = r(N) \in \N,$ $\sigma=\sigma(c)>0$ are chosen in such a way that $$\int^{0}_{-r}N(s)ds> 0.99, \quad 
2c\frac{e^{\lambda(c) \sigma} -1}{e^{c \sigma} -1} < 0.01.$$ 
\end{itemize}
Obviously,  such $U: [2\sqrt{G(0)},+\infty) \times \mathcal{S}  \to (0,+\infty)$ is locally continuous at each $(c_0,N_0)$.  For example, $U(c,N)$ can be considered as a constant (hence, continuous) function in some small neighborhood of $(c_0,N_0) \in [2\sqrt{G(0)},+\infty) \times  \mathcal{S}$ satisfying $\int_{0}^{+\infty}N_0(s)ds =0$.

Clearly, if $\phi(t) \in (0,1]$ for all $t \in \R$, then inequality (\ref{iu}) is true because   $U(c,N)\geq 1$. In particular, this happens if 
the profile  $\phi(t)$ is {nondecreasing}  and  $2\beta\int_{-\infty}^0N(s)ds > 1$, see Remark \ref{ir}.

Thus let us suppose that $\phi(t_0) > 1$ at some point $t_0$.  Then at least one of the following three possibilities can occur: 

\noindent \underline{Situation I}. Solution $\phi(t)$ is {nondecreasing}  and  $\int_{-\infty}^0N(s)ds = 0$ (so that  $\int_{0}^{+\infty}N(s)ds =1$).  In such a case, by Lemma \ref{li}, there is  some finite $T_1$ such that 
$\phi(+\infty)=\phi(T_1)=2\beta$ and $(N*\phi)(T_1) \leq 1$. For $x(t):= -\ln\phi(t)$, we have  $ x'(t) = -\phi'(t)/\phi(t) \geq -\lambda(c)$  for all $t\in \R$ 
and 
$$
\int_{0}^{+\infty}e^{-x(T_1-s)}N(s)ds \leq (N*\phi)(T_1)= \int_{\R}e^{-x(T_1-s)}N(s)ds \leq 1.
$$
Now, set $m:= \min_{s\in \R} x(s)$ and observe that $x(t) = x(T_1)-\int^{T_1}_t x'(s)ds \leq m -\lambda(c)(t-T_1)$ for $t \leq T_1$.  Thus 
$$
\int_{0}^{+\infty}e^{-m-\lambda(c)s}N(s)ds \leq \int_{0}^{+\infty}e^{{-}x(T_1-s)}N(s)ds \leq 1
$$
and therefore 
$$
\frac{1}{2\beta} = \frac{1}{\phi(T_1)} = e^m \geq \int_{0}^{+\infty}e^{-\lambda(c)s}N(s)ds. 
$$
Thus we can take 
$$
2\beta= \phi(T_1)  \leq \left(\int_{0}^{+\infty}e^{-\lambda(c)s}N(s)ds\right)^{-1} = U(c,N).
$$
{The latter shows that  Situation I cannot occur if $2\beta > U(c,N)$.}

\vspace{2mm}
 
\noindent  \underline{Situation II}. Solution $\phi(t)$ is not {nondecreasing} 
 { and  $\int_{0}^{+\infty}N(s)ds>0.$} 
Then we can repeat the above arguments to conclude that, for the local maxima $\phi(t_j) >1$ of $\phi$ we have that 
 $$
\sup_{t \in \R}\phi(t) =\sup_j \phi(t_j)  \leq \left(\int_{0}^{+\infty}e^{-\lambda(c)s}N(s)ds\right)^{-1}= U(c,N). 
 $$
\underline{Situation III}. 
{Solution $\phi(t)$ is not {nondecreasing} 
  and $\int_{0}^{+\infty}N(s)ds=0$.} 
Suppose, on the contrary,  that $\phi(t_0) > U(c,N) = 2\exp{(\lambda(c) (r+\sigma))}$ for some $t_0$.  Then $\phi(t) \geq 2$ on some maximal closed interval $[a,b] \ni t_0$.  We claim that $b-a \geq r+\sigma$.  Indeed, otherwise, since $\phi'(t) \leq \lambda(c) \phi(t), $ $\phi(a) =2,$ $ t_0-a < b- a,$  we get the following contradiction 
$$
\phi(t_0) \leq \phi(a) e^{\lambda(c) (t_0-a)} < 2e^{\lambda(c) (b-a)} \leq   2e^{\lambda(c) (r+\sigma)}. 
$$
In consequence, 
$$
(N*\phi)(t)  \geq  \int_{-r}^0 \phi(t-s)N(s)ds \geq 1.98, \quad t \in [a, a+\sigma], 
$$
so that ${G((N*\phi)(t)) < 0}, \ t \in [a,a+\sigma],$ and $\phi''(t) -c \phi'(t) > 0,$ $\ t \in [a,a+\sigma]$.  In particular, $\phi'(t) > e^{c(t-a)}\phi'(a)$ for all $a  <t \leq a+\sigma$ and thus 
$$
2e^{\lambda (t-a)}  \geq \phi(t) > 2+ \frac 1c \{e^{c(t-a)}-1\}\phi'(a), \quad a  <t \leq a+\sigma. 
$$
Therefore 
$$
2e^{\lambda \sigma} > 2+ \frac 1c \{e^{c\sigma}-1\}\phi'(a), 
$$
so that $0 \leq \phi'(a) <0.01$.  Next, let $[a_-, b_+] \supseteq [a,b]$ be the maximal interval 
where $\phi(t) \geq 1.1$. Then, for all $t \in [a_-, a + \sigma]$, we have $\phi''(t)-c\phi'(t) >0$ since
$$
(N*\phi)(t) \geq  \int_{-r}^0 \phi(t-s)N(s)ds \geq 0.99\cdot 1.1 > 1, \ t \in [a_-, a+\sigma]. 
$$
But then 
$$
\phi'(t) < \phi'(a)e^{c(t-a)} < 0.01e^{c(t-a)}, \ t \in [a_-,a);
$$
$$
\phi(t) > 2- \frac{0.01}{c}\{1-e^{c(t-a)}\} > 2-  \frac{0.01}{c} > 1.9,  \quad t \in [a_-,a), 
$$
a contradiction (since $\phi(a_-) =1.1$). 
\qed
\end{proof}
\begin{cor}
\label{twoe} Assume 
that $G(s) <G(0)=G^*$ for all $s >0$ and let some $c\geq 2\sqrt{G(0)}$ be fixed.  Then for each sufficiently large $\beta >1$  equations (\ref{twe2m}) and (\ref{twe2an}) share the same set of semi-wavefronts propagating at the speed $c$. 
\end{cor} 
\begin{proof} Due to Lemma \ref{ogran} and the definition of $g_{\beta}(u)$, it suffices to take
$\beta > U(c,N_c)$. \qed
\end{proof}


\section{Existence of semi-wavefronts for $c\geq 2\sqrt{G^*}$} \label{ESW}
\subsection{Proof of Theorem \ref{Te1} in the non-critical case and without the Allee effect.}

\vspace{1mm}

\noindent In this section, we are going to prove Theorem \ref{Te1} in the case when $G^*=G(0) > G(u)$ for all $u > 0$.  By the first assumption of (\ref{DG}), 
 there exists  some positive $p\geq G(0)$ such that 
$$
G(u) \geq G_\#(u):= G(0)- pu, \quad u \in [0,2\beta]. 
$$
From Lemma \ref{bebe}, we know that  the condition 
$c\geq 2\sqrt{G(0)}$ is necessary for the existence of semi-wavefronts. Thus we have to prove only the sufficiency of this inequality. 

First, consider    $$r(\phi)(t):= b\phi(t) + g_{\beta}(\phi(t)){G((N_c*\phi)(t))},$$ where  $g_{\beta}(u)$ is defined by (\ref{G}), $\beta$ is as in Corollary \ref{twoe}, and $b>G(0)- \min_{[0,2\beta]}G(u)$.   In view of Corollary \ref{twoe}, it suffices to establish that the equation  
\begin{equation}
\label{twe2mm} 
\phi''(t) - c\phi'(t)  - b \phi(t) + r(\phi)(t)=0
\end{equation}
has a semi-wavefront.
Observe that if  a continuous function $\psi(t), \ 0 \leq \psi(t)\leq 2\beta,$ satisfies $\psi(s) \leq \beta$ at some point $s \in \R$, then  
\begin{equation}\label{Ar}
r(\psi)(s)= \psi(s)\left[b + G((N_c*\psi)(s))\right] \geq {\psi(s)(b +  \min_{[0,2\beta]}G(u))} \geq \psi(s)G(0)  \geq 0. 
\end{equation}
Now, if   $\beta\leq \psi(s) \leq 2\beta$, then  
$$
r(\psi)(s)=  { b\psi(s)+  (2\beta -\psi(s))G((N_c*\psi)(s))} \geq b\psi(s)+  (2\beta -\psi(s)) \min_{[0,2\beta]}G(u)
$$
\begin{equation} \label{Br}
\geq \beta (2 \min_{[0,2\beta]}G(u) + b- \min_{[0,2\beta]}G(u))> G(0)\beta >0.  
\end{equation}
Furthermore,  the inequality $0 \leq \psi(s)\leq 2\beta,\ s \in \R,$ implies  that 
\begin{equation}\label{Ac}
 r(\psi)(s) =  b\phi(s) + g_{\beta}(\phi(s)){G((N_c*\phi)(s))} \leq 2\beta(b+G(0)). 
\end{equation}
Next, we consider the non-delayed KPP-Fisher equation  $u_t= u_{xx} +G(0)g_{\beta}(u)$.  The profiles $\phi$
of the travelling fronts $u(x,t)= \phi (x+ct)$ for this equation satisfy
\begin{equation}\label{kppm}
\phi''(t) - c\phi'(t)  + G(0)g_{\beta}(\phi(t))=0, \ c \geq 2\sqrt{G(0)}.
\end{equation}
Recall that  $0< \lambda \leq \mu$ denote eigenvalues of equation (\ref{kppm}) linearized around $0$ (i.e. 
$\chi(\lambda) = \chi(\mu) =0$ where $\chi(z) := z^2-cz+G(0)$). In the sequel,  $\phi_+(t)$ will denote the unique monotone front to (\ref{kppm})  normalised (cf. \cite[Theorem 6]{GT})  by the condition $$ \phi_+(t)= (-t)^je^{\lambda t} (1+o(1)), \ t \to -\infty, \ j \in \{0,1\}.$$
Let us note here that  $\phi_+(t)$ satisfies the linear differential equation 
$$
\phi''(t) - c\phi'(t)  + G(0)\phi(t)=0
$$
for all $t$ such that $\phi_+(t) < \beta.$
In particular, if $c >2\sqrt{G(0)}$  then there exists (see e.g. \cite[Theorem 6]{GT}) $C \geq 0$ such that  
\begin{equation} \label{1mo}
 \phi_+(t)= e^{\lambda t}  - C e^{\mu t}, \ t \leq \phi_+^{-1}(\beta).   \end{equation}
 Let $z_1< 0< z_2$ be the  roots of the equation  $z^2 -cz-b =0$. 
Set $z_{12} =z_2-z_1>0$  and consider the  integral operator $\mathcal A$ depending on $b$ and defined by
$$({\mathcal A}\phi)(t) = \frac{1}{z_{12}}\left\{\int_{-\infty}^te^{z_1
(t-s)}r(\phi)(s)ds + \int_t^{+\infty}e^{z_2
(t-s)}r(\phi)(s)ds \right\}.
$$
\begin{lem}  \label{usl} Assume that  $b>G(0)- \min_{[0,2\beta]}G(u)$ and let  $0\leq \phi(t) \leq \phi_+(t)$, then 
$$0 \leq ({\mathcal A}\phi)(t) \leq \phi_+(t).$$
\end{lem}
\begin{proof} The lower estimate is obvious since $0\leq \phi(t) \leq \phi_+(t) \leq 2\beta$ and therefore
$r(\phi)(t) \geq 0$ in view of {(\ref{Ar}) and (\ref{Br}). Now, since  $\phi(t) \leq \phi_+(t)$ }and $bu+G(0)g_{\beta}(u)$ is an increasing function, we find that 
$$
r(\phi)(t) \leq b\phi(t) + G(0)g_{\beta}(\phi(t)) \leq b\phi_+(t) + G(0)g_{\beta}(\phi_+(t)) =:R(\phi_+(t)). 
$$
Thus 
$$
({\mathcal  A}\phi)(t) \leq  \frac{1}{z_{12}}\left\{\int_{-\infty}^te^{z_1
(t-s)}R(\phi_+(s))ds + \int_t^{+\infty}e^{z_2
(t-s)}R(\phi_+(s))ds \right\}= \phi_+(t),
$$
and the lemma is proved. \qed
\end{proof}
Lemma \ref{usl} says that $\phi_+(t)$ is an {\it upper} solution for (\ref{twe2mm}), cf. \cite{wz}. Still,  we need to find a {\it lower} solution.  
Here, assuming that $c >2\sqrt{G(0)}$  and that $N_c$ has a compact support we will use the following  well known  ansatz (see e.g. \cite{wz})  
$$
\phi_-(t)= \max\{0,e^{\lambda t} (1- Me^{\epsilon t})\},
$$
where $\epsilon \in (0, \lambda)$  and $M \gg 1$  are chosen in 
such a way that $$- \chi(\lambda+\epsilon)   >(Lp/M) \int_{-\infty}^{\infty} e^{-\epsilon s} N_c(s) ds$$ (here  $L : = \sup_{t \in \R} \phi_+(t)e^{-\epsilon t}$),  $\lambda +\epsilon < \mu$, and
$$0< \phi_-(t) < \phi_+(t) < e^{\epsilon t}< 1, \quad t < T_c, \ \mbox{where} \ \phi_-(T_c) =0.$$
The above inequality $\phi_-(t) < \phi_+(t)$ is possible due to  the representation (\ref{1mo}). We also  note  that $(N_c*\phi_+)(t) \leq L e^{\epsilon t}\int_\R e^{-\epsilon s}N_c(s)ds$. 
\begin{lem}  \label{nre} Assume that {$c>2\sqrt{G(0)}$}, { $N_c$ has a  compact support}, $b>2p\beta$. Then the inequality  $\phi_-(t) \leq \phi(t) \leq \phi_+(t), \ t \in \R,$ implies that 
\begin{equation}\label{ul2}
\phi_-(t) \leq ({\mathcal  A}\phi)(t) \leq \phi_+(t), \quad t \in \R. 
\end{equation}
\end{lem}
\begin{proof} 
Due to Lemma \ref{usl}, it suffices to prove the first inequality in (\ref{ul2}) for $t \leq T_c$. 
Since $0< \phi(t) < 1 < \beta, \ t \leq T_c$, we have, for $t\leq T_c $, that 
$$
({\mathcal A}\phi)(t) \geq   \frac{1}{z_{12}}\left\{\int_{-\infty}^te^{z_1
(t-s)}r(\phi)(s)ds + \int_t^{T_c}e^{z_2
(t-s)}r(\phi)(s)ds \right\}
$$
$$
=\frac{1}{z_{12}}\left\{\int_{-\infty}^te^{z_1(t-s)} \phi(s)[b +{G((N_c*\phi)(t))}] ds + 
\int_t^{T_c}e^{z_2(t-s)}\phi(s)[b +{G((N_c*\phi)(t))}] ds \right\}
$$
$$ \geq \frac{1}{z_{12}}\left\{\int_{-\infty}^te^{z_1
(t-s)} \Gamma(s) ds + \int_t^{T_c}e^{z_2
(t-s)}\Gamma(s) ds \right\} 
$$
$$ =\frac{1}{z_{12}}\left\{\int_{-\infty}^te^{z_1
(t-s)} \Gamma(s) ds + \int_t^{+\infty}e^{z_2
(t-s)}\Gamma(s) ds \right\} =: Q(t),
$$
where $\Gamma(s): = \phi_-(s)[b +{G_\#((N_c*\phi_+)(t))}]$. 
To  estimate  $Q(t)$, we  find, for  $t \leq T_c$, that
$$
\phi_-''(t) - c\phi_-'(t)  -b\phi_-(t) + b\phi_-(t) + \phi_-(t){G_\#((N_c*\phi_+)(t))} 
$$
$$
=- \chi(\lambda+\epsilon) M e^{(\lambda +\epsilon)t} - {[G(0)-G_\#((N_c*\phi_+)(t))]}e^{\lambda t} (1- Me^{\epsilon t}) 
$$
$$
\geq- \chi(\lambda+\epsilon) M e^{(\lambda +\epsilon)t} - e^{\epsilon t }e^{\lambda t} Lp { \int_{-\infty}^{\infty} e^{-\epsilon s} N_c(s) ds}$$
$$ = M e^{(\lambda +\epsilon)t}\left(
- \chi(\lambda+\epsilon)   - {\frac{Lp}{M} \int_{-\infty}^{\infty} e^{-\epsilon s} N_c(s) ds}\right) >0. 
$$
But then, rewriting the latter differential inequality in the equivalent  integral form  (see e.g. \cite[Lemma 18]{TPT}) and using  the fact that 
$$ \phi'_-(T_c^+) - \phi'_-(T_c^-) = - \phi'_-(T_c^-) >0, $$
we can conclude that $Q(t) \geq \phi_-(t), \ t \in \R$.  Hence, $({\mathcal  A}\phi)(t) \geq  \phi_-(t), \ t \in \R$. \qed
\end{proof}

Next,  with some $\rho >0$,  we will consider the  Banach
space
\[
C_{\frak{m}} = \{y \in C(\R, \R): |y|_{\frak{m}}: = \sup_{s \leq
0}e^{-0.5\lambda s} |y(s)|+  \sup_{s \geq
0}e^{-\rho s} |y(s)| < + \infty  \},
\]
\[
C^1_{\frak{m}} = \{y \in C_{\frak{m}}: y' \in C_{\frak{m}}, \
|y|_{1,\frak{m}}: =  |y|_{\frak{m}} +  |y' |_{\frak{m}} < +\infty
\}. \]
In order to establish the existence of semi-wavefronts to equation (\ref{twe2mm}), it suffices to prove that 
the equation ${\mathcal  A}\phi=\phi$ has at least one  solution  from the set 
$$\frak{K} = \{x \in C_{\frak{m}}: \phi_-(t) \leq x(t) \leq  \phi_+(t), \ t \in \R\}.$$  Note that $\phi_+(t)e^{-\rho t} = O(e^{-\rho t})$ at $+\infty$ and 
$\phi_+(t)e^{-\lambda t/2} = O(|t|e^{\lambda t/2})$ at $-\infty$,  so that the norm $|\phi_+|_{\frak{m}}$ is finite.  Since $0 \leq x(t) \leq  \phi_+(t)$ implies 
$|x|_{\frak{m}}\leq |\phi_+|_{\frak{m}}$, the set $\frak{K}$ is bounded and non-empty. 
Observe also that 
the convergence $x_n \to x$ in $\frak{K}$ is equivalent to the
uniform convergence  on compact subsets of
$\R$.
\begin{lem} \label{ple} Let $c >2\sqrt{G(0)}$. Then $\frak{K}$ is a non-empty, closed, bounded and convex subset of $C_{\frak{m}}$
and ${\mathcal  A}:\frak{K} \to \frak{K}$ is completely continuous. As a consequence, the integral equation
${\mathcal  A}\phi=\phi$  has at least one  positive bounded solution in $\frak{K}$. 
\end{lem}
\begin{proof} For the above mentioned properties of ${\mathcal  A}:\frak{K} \to \frak{K}$ see, for example, the proof of Lemma 11 in \cite{HK}. Then the existence of at least one solution 
$\phi \in \frak{K}$ to the equation ${\mathcal  A}\phi=\phi$  is  an immediate  consequence  of the 
 Schauder fixed point theorem.   \qed  \end{proof}


\subsection{Proof of Theorem \ref{Te1} in the general case}

\vspace{1mm}
{In what follows,  $C_b:=C_b(\mathbb{R},\mathbb{R}^N)$ will denote the space of all continuous
and bounded functions from $\mathbb{R}$ to $\mathbb{R}^N$, with the
supremum norm $|y|_{\infty}=\sup _{s\in \mathbb{R}}|y(s)|$. }

\begin{thm} \label{34} Assume that  $c\geq 2\sqrt{G^*}$. Then the integral equation
${\mathcal  A}\phi=\phi$  has at least one  positive bounded solution in $\frak{K}$.
\end{thm}
\begin{proof}  Assume first that $K$ (hence, $N_c(s)$ for each $c>0$) has a  compact support and $G(0) = G^* > G(u)$ for $u >0$. 
If $c>2\sqrt{G(0)}$ then  the assertion of the theorem follows from Lemma \ref{ple}. 

It remains to analyse the case when  $c=2\sqrt{G(0)}.$ Consider the sequence $c_j:= c+1/j$. Since $c_j >2\sqrt{G(0)}$,   there exists a semi-wavefront $\phi_j$ of equation (\ref{twe2mm}) for each $j$, which  we can 
normalise by the condition $\phi_j(0)= 1/2 = \max_{s \leq 0} \phi_j(s)$. It is easy to see that the set 
$\{\phi_j, j \geq 0\}$ is precompact in the compact-open topology of {$C_b(\R,\R)$}  and therefore we can also assume that $\phi_j \to \phi_0$ uniformly on compact subsets of $\R$, where $\phi_0 \in C_b(\R,\R)$ and 
$\phi_0(0) =1/2= \max_{s \leq 0} \phi_0(s)$.    In addition, 
$R_j(s):= r(\phi_j)(s) \to R_0(s):= r(\phi_0)(s)$ for each fixed $s \in \R$. The sequence 
$\{R_j(t)\}$ is also uniformly bounded on $\R$, see (\ref{Ac}). All this allows us to apply Lebesgue's dominated convergence theorem in  
\begin{equation}\label{RLi}
\frac{1}{\epsilon_j'}\left\{\int_{-\infty}^te^{z_{1,j}
(t-s)}R_j(s)ds + \int_t^{+\infty}e^{z_{2,j}
(t-s)}R_j(s)ds \right\} = \phi_j(t), 
\end{equation}
where $z_{1,j} <0< z_{2,j}$ satisfy $z^2-c_jz -b =0$ and $\epsilon_j':= z_{2,j}- z_{1,j}$. 
Taking the limit in (\ref{RLi}), we obtain that ${ \mathcal  A}\phi_0 = \phi_0$ with $c=2\sqrt{G(0)}$ and therefore $\phi_0$ is a non-negative solution 
of equation (\ref{twe2an}) satisfying the condition $\phi_0(0)=1/2= \max_{s \leq 0} \phi_0(s)$.  Lemma \ref{po} shows 
that actually $\phi_0(t)>0$ for all $t\in \R$.  We claim, in addition, that $\inf_{s \leq 0}\phi_0(s) =0$ and therefore 
$\phi_0(-\infty)=0$ in view of  Lemma  \ref{bebe}. Indeed, otherwise there exists a positive $k_0$ such that 
$k_0\leq  \phi_0(t) \leq 1/2$ for all $t \leq 0$. This implies immediately 
that 
$$0< 0.5k_0\min_{u \in [k_0,1/2]}{G(u)} \leq  \phi_0(t)G((N_c*\phi_0)(t)) \leq \max_{u \in [k_0,1/2]}{G(u)}$$ for all sufficiently large negative $t$ (say, for $t \leq t_0$).  But then 
$$
\phi_0'(t) = \phi_0'(t_0) + c(\phi_0(t)-\phi_0(t_0)) +\int^{t_0}_t\phi_0(s)G((N_c*\phi_0)(s))ds \to +\infty \ \mbox{as} \ t \to -\infty, 
$$
contradicting the positivity of $\phi_0(t)$.  In consequence,  $\phi_0$ is a semi-wavefront for $c=2\sqrt{G(0)}.$

Next, consider the case when  $K$ has a  compact support with  $K(0,0)>0$ (hence,  $N_c(s)$ has a  compact support with  $N_c(0)>0$ for each $c>0$) and when $G(0) <  G^*$, $c > 2\sqrt{G^*}$. For each $j \geq 2$, we define a
continuous function $G_j:\R_+ \to \R$ with $G_j(0)= G^*+1/j$  which  coincides with $G(u)$ on the interval $[1/j,+\infty)$ and is linear on $[0,1/j]$. 
Clearly, each $G_j$ satisfies all conditions of the previous subsection and for every positive $A$ there exists integer $j_0$ such that $G_j(u)=G(u)$ for all $u \geq A, \ j \geq j_0$. 
Again, we  know that for each large $j$ there exists a semi-wavefront $\phi_j$ of the equation 
\begin{equation}
\label{twe2mmj} 
\phi''(t) - c\phi'(t)  - b \phi(t) + r_j(\phi)(t)=0,
\end{equation}
where 
 $$r_j(\phi)(t):= b\phi(t) + g_{\beta}(\phi(t)){G_j((N_{c}*\phi)(t))}.$$  
We will 
normalise $\phi_j$ by the condition $\phi_j(0)= 1/2 = \max_{s \leq 0} \phi_j(s)$. It is easy to see that the set 
$\{\phi_j, j \geq 0\}$ is precompact in the compact-open topology of $C_b(\R,\R)$  and therefore we can also assume that $\phi_j \to \phi_*$ uniformly on compact subsets of $\R$, where 
$\phi_*(0) =1/2= \max_{s \leq 0} \phi_0(s)$.    In addition, 
$R_j(s):= r_j(\phi_j)(s) \to R_*(s):= r(\phi_*)(s)$ for each fixed $s \in \R$. Indeed,  suppose that   $(N_c*\phi_* )(s)> 0$ for some $s \in \R$, then $(N_c*\phi_j)(s)> 0$ for all large $j$  so that  $G_j((N_c*\phi_j)(s))= G((N_c*\phi_j)(s)) $ if $j$ is sufficiently large. In consequence,  $\lim_{j \to +\infty} R_j(s) = R_*(s)$. On the other hand, if $(N_c*\phi_*)(s)= 0$ then necessarily $\phi_*(s) =0$ (recall that $N_c(0)>0$) and therefore $\phi_j(s) \to 0$ as $j \to +\infty$.  Thus 
$$
\lim_{j \to + \infty}\left[b\phi_j(s) + \phi_j(s){G_j((N_c*\phi_j)(t))}\right] = b\phi_*(s) = b\phi_*(s) + \phi_*(s){G_*((N_c\phi_*)(s))}. 
$$

The sequence 
$\{R_j(t)\}$ is also uniformly bounded on $\R$. All this allows us to apply Lebesgue's dominated convergence theorem in (\ref{RLi}) and conclude 
 that ${ \mathcal  A}\phi_* = \phi_*$  and therefore {$\phi_*$} is a non-negative solution 
of equation (\ref{twe2an}) satisfying the condition ${\phi_*(0)}=1/2= \max_{s \leq 0}{\phi_*(s)}$. Arguing as above,  we conclude that  {$\phi_*$} is a semi-wavefront propagating 
with the speed $c > 2\sqrt{G^*}$. The limiting case when $c= 2\sqrt{G^*}$,  $G(0) <  G^*$ and 
  $K$ has a  compact support with  $K(0,0)>0$, 
 can be analyzed in the same way as it was done in the second paragraph of this proof.

Finally, in order to prove the theorem for general kernels, we can use a similar limit argument by constructing a sequence of compactly supported  kernels $N_j$ converging monotonically to $N_c$.  Indeed, set  $$N_j (s) = \left[N_c(s)+  (2j)^{-1}\left(\int_{-\infty}^{-j} N_c(s) ds + \int_j^{\infty} N_c(s) ds\right)\right]\left(1-\frac1 j\right) + \frac 1 j $$
  for $s \in [-j,j]$, and set $N_j (s) = 0 $ otherwise.  
Clearly, $N_j (0)>0$. Therefore,  as we have already proved,  for each  fixed $c \geq 2\sqrt{G^*}$ and $N_j$   there exists  a semi-wavefront  $\phi_j$ propagating with the velocity $c$ and satisfying the condition $\phi_j(0)= 1/2 = \max_{s \leq 0} \phi_j(s)$. 
Due to Lemma \ref{ogran}, $0< \phi_j(t) \leq U(c,N_j)$ for all $t \in \R$. 
By using the explicit form of $U(c,N_j)$ given in Lemma \ref{ogran}, it is easy to show that the sequence 
$\{\phi_j(t)\}$ is uniformly bounded on $\R$.  The sequence 
$\{\phi'_j(t)\}$ is uniformly bounded on $\R$ as well, so  we can assume that 
 $\phi_j \to \phi_0 \in C_b(\R,\R)$ uniformly on compact subsets of $\R$. But then also 
$\phi_0(0) =1/2= \max_{s \leq 0} \phi_0(s)$ so that, arguing as in the first part of our proof, we conclude that $\phi_0(x+ct)$ must be a semi-wavefront for equation (\ref{twe2an}) with a general kernel. 
\qed
\end{proof}


\section{Monotone wavefronts: the uniqueness}

We will assume {in the whole section that   $c \geq 2\sqrt{G(0)}$, all considered  wavefronts are monotone and that   there exists a finite derivative 
$G'(1) <0$. 
Then the function 
$H(u): = G(u)/(u-1), \ u \not=1,\ H(1) = G'(1)$, is well defined and  continuous.  Set $H_*= \min_{u\in [0,1]}H(u)$, $H^*= \max_{u\in [0,1]}H(u)$, clearly,  $H_* \leq H^*<0$. 
Furthermore, the kernel $K:\R_+\times\R \to \R_+$ will  satisfy
$$
\frak{I}_{c}(z):= \int_{\R_+\times\R}  K(s,y)e^{-z(cs+y)}dsdy \in \R, \ \  \frak{I}_{c}(0)=1, \ \lim_{z\to- \lambda_0(c)^+}\frak{I}_{c}(z)=+\infty, 
$$
for all $z \in (-\lambda_0(c), \lambda_1(c)), \ c >0$, and some  $\lambda_0(c), \lambda_1(c) \in (0,+\infty]$.}

 Again,  we will assume that 
$\frak{I}_{c}(\lambda)$ is a scalar continuous function of variables  $c, \lambda$.  Next, consider the characteristic function {of equation (\ref{twe2an}) linearized around the positive equilibrium}
$$\chi_+(z,c) =z^2-cz+G'(1)\int_{\R}e^{-z{s}}N_c(s)ds = z^2-cz+G'(1)\frak{I}_c(z).$$
Observe that $\chi_+(0,c) <0,$ $\chi_+(-\lambda_0(c)^+,c) =-\infty,$ and  $\chi^{(4)}_+(z,c) <0, \ z \in  (-\lambda_0(c), \lambda_1(c))$, so that  the  function  
$\chi_+(z,c)$ has at  most four negative zeros. 
\subsection{About the existence of negative zeros of  $\chi_+(z,c)$}\label{S22}
Suppose first that  supp\,$N_c$  belongs to $(-\infty,0]$.  Then the characteristic function $\chi_+$ takes the form  
$$
\chi_+(z,c) = z^2-cz +G'(1) \int_{-\infty}^0 N_c(v)e^{-z v}dv,
$$
so that $
\chi_+(0,c)<0$ and $\chi_+(-\infty,c)=+\infty$. In consequence, $\chi_+(z,c)$ has at least one negative zero if  supp\,$N_c \subset (-\infty,0]$.  

Next, suppose that  supp\,$N_c \cap (0, +\infty) \not=\emptyset$. Let  $\phi(t)$ be a positive monotone wavefront. Then there exists  $r_0 >0$ such that for each $s \in \R$ and 
positive $\nu = \int^{+\infty}_{r_0}N_c(r)dr \in (0,1)$, it holds
\begin{eqnarray}
\label{1N}
& & 1-(N_c*\phi)(s) = \int_\R(1-\phi(s-r))N_c(r)dr   \nonumber \\
& & \geq \int^{+\infty}_{r_0}(1-\phi(s-r))N_c(r)dr\geq \nu (1-\phi(s-r_0)).
\end{eqnarray}
Our subsequent analysis is inspired by the arguments proposed in \cite{FZ} and \cite{HT},  we present them here for the sake of completeness. 
Set $y(t)=1-\phi(t)$ {in (\ref{twe2an})}. Then 
\begin{equation}\label{efy}
y''(t)-cy'(t) + G(0)y(t) = G(0)y(t) - \left[\phi(t)H((N_c*\phi)(t))\right](N_c*y)(t), \quad t \in \R, 
\end{equation}
so that (cf.  \cite{FZ,GT})
$$
y(t) = \int_0^{+\infty}A(s)\left( G(0)y(t+s) - \left[\phi(t+s)H((N_c*\phi)(t+s))\right](N_c*y)(t+s)\right)ds,
$$
where $$ A(s) = \left\{
\begin{array}{ccc}
a'(e^{-\lambda(c) s}- e^{-\mu(c) s}),& s \geq 0\  \mbox{and}\ c >2\sqrt{G(0)},\\
a'se^{- cs/2},& s\geq 0\  \mbox{and}\ c =2\sqrt{G(0)},\\
0 ,& s \leq 0, \ c \geq 2\sqrt{G(0)}, 
\end{array} \right. $$
with positive $a'$  chosen to assure the normalization condition $\int_\R A(s)ds =1/G(0)$.

Set $\nu^* = \phi(0)|H^*|\nu, \ \hat \nu =  \nu^*\int_0^{r_0/2}A(s)ds$.  In view of (\ref{1N}), this implies that, for all $t \geq 0$,
$$
y(t) \geq  \phi(0)|H^*|\int_0^{+\infty}A(s)\nu y(t+s-r_0)ds  \geq \nu^* \int_0^{r_0/2}A(s) y(t+s-r_0)ds\geq \hat \nu y(t-r_0/2).
$$
Therefore, for some $C>0$ and $\sigma= 2r_0^{-1}\ln \hat \nu <0$,  
\begin{equation}\label{rha}
y(t) \geq Ce^{\sigma t}, \quad t \geq 0. 
\end{equation}
Hence, 
\begin{equation}\label{si}
0 \geq \sigma_* = \liminf_{t\to +\infty} \frac 1 t \ln y(t) \geq \sigma.  
\end{equation}
Suppose, on the contrary, that $\chi_+(z,c)$ does not have  negative zeros. Then $\chi_+(\sigma_*,c)<0$ (we admit here the situation when $\chi_+(\sigma_*,c)=-\infty$), so that there exist a large $n_0>0$ and small $\delta, \rho >0$ such that 
$$
q:=\inf_{x \in (\sigma_*-\delta, \sigma_*+\delta)} \frac{G(0) -  (1-\rho)\left(\max_{u\in [1-\rho,1]}H(u)\right)\int_{-n_0}^{n_0} N_c(v)e^{-{x}  v}dv}{x^2-cx +G(0)}  >1.
$$
Next, let $\mu_* \in  (\sigma_*-\delta, \sigma_*)$ be such that 
$$\mu_*+ n_0^{-1}\ln q > \sigma_*. $$
Clearly, for some $C_1 >0$, it holds  that 
$$
y(t) \geq C_1e^{\mu_*  t}, \quad t \geq -n_0. 
$$
Since $y(+\infty) =0$, there exists $T'>0$ such that  $(N_c*y)(s) < \rho, \ \phi(s) > 1 - \rho$ for all $s \geq T'$, then $(N_c*y)(t+s) < \rho,\ \phi(t+s) > 1 - \rho $ for all $t \geq T', s \geq 0$ and 
$$
y(t) = \int_0^{+\infty}A(s)\left( G(0)y(t+s) - \left[\phi(t+s)H((N_c*\phi)(t+s))\right](N_c*y)(t+s)\right)ds
$$
$$
 \geq \int_0^{+\infty}A(s)(G(0)y(t+s)- (1-\rho)\max_{u\in [1-\rho,1]}H(u)(N_c*y)(t+s))ds  
$$
$$
\geq C_1 e^{\mu_* t} \int_0^{+\infty}A(s)e^{\mu_* s}\left(G(0) -  (1-\rho)\max_{u\in [1-\rho,1]}H(u)\int_{-n_0}^{n_0} N_c(v)e^{-\mu_*  v}dv \right)ds
$$
$$
=C_1 \frac{e^{\mu_*  t}}{\mu_*^2-c\mu_* +G(0)} \left(G(0) -  (1-\rho)\max_{u\in [1-\rho,1]}H(u)\int_{-n_0}^{n_0} N_c(v)e^{-\mu_*  v}dv \right) \geq C_1q e^{\mu_*  t},  \quad t \geq T'. 
$$
Repeating the same argument for $y(t)$ on the interval $[T'+n_0,+\infty)$, we find similarly that $y(t) \geq C_1q^2 e^{\mu_*  t},  \ t \geq T'+n_0.$  Reasoning in this way, 
we obtain the estimates
$$y(t) \geq C_1q^{j+1} e^{\mu_*  t}\geq C_1e^{\mu_* T'}e^{(\mu_*+n_0^{-1}\ln q)(t-T')},  \ t \in [T'+n_0j,T'+n_0(j+1)], \quad j=0,1,2\dots 
$$
This yields the following contradiction:  
$$
\sigma_* = \liminf_{t\to +\infty} \frac 1 t \ln y(t) \geq \mu_*+n_0^{-1}\ln q > \sigma_* . 
$$
As a product of the above reasoning, we also get the following statement: 
\begin{lem}\label{bp} Suppose that supp\,$N_c \cap (0, +\infty) \not=\emptyset$ and let $\phi(t)$ be a positive monotone wavefront.  Set $y(t)=1-\phi(t)$ and let $\sigma_*$ be defined as in (\ref{si}). Then $\chi_+(\sigma_*,c)$ is finite and  non-negative. In particular,  $-\lambda_0(c) < \sigma_* <0$ and  $\chi_+(z,c)$ has at least one zero on the interval $[\sigma_*,0]$. 
\end{lem}
\begin{remark} \label{Re1} Lemma \ref{USA} below  further   improves the result of Lemma \ref{bp}.  Next, let $\{z: \Re z > \alpha(y)\}\subset \C$ be the maximal open  strip where the Laplace transform $\tilde y(\lambda)$ of $y(t)$ is defined. Since $y(t)$ is bounded on $\R$, we have that $ \alpha(y) \leq 0$. On the other hand, by the definition of $\sigma_*$,  it is easy to see 
that $\lim_{t\to+\infty}y(t)e^{-\lambda t} =+\infty$ for every $\lambda < \sigma_*$. Thus $ \alpha(y) \geq  \sigma_* > -\lambda_0(c)$. Note also that $\alpha(y)$ is a singular point 
of $\tilde y(\lambda)$. 
\end{remark}

\subsection{Three other auxiliary results}
\begin{lem}\label{USAI} Suppose that supp $N_c \cap (0,+\infty)=\emptyset$ and  let $\phi(t)$ be a monotone wavefront to  the equation (\ref{twe2an}). Then $y(t)=1-\phi(t)$
satisfies 
\begin{equation}\label{areI}
y(t) \geq y(s)e^{(s-t)|H_*|/c} \quad \mbox{for all} \ t \geq s, \ t,s \in \R. \end{equation}
\end{lem}
\begin{proof} By (\ref{efy}),  we have that 
$$
y''(t)-cy'(t) + \left[\phi(t)H((N_c*\phi)(t))\right](N_c*y)(t)=0, \quad t \in \R. 
$$
Clearly, $(N_c*\phi)(t) = \int_{-\infty}^0\phi(t-s)N_c(s)ds \geq \phi(t)\int_{-\infty}^0N_c(s)ds=\phi(t), \ t \in \R$, so that 
$$
y(t) = 1-\phi(t) \geq 1-(N_c*\phi)(t)= (N_c*y)(t), \quad   -\phi(t)H((N_c*\phi)(t)) \leq -H_*,\quad t \in \R. 
$$
Using the notation
\begin{equation}\label{rr}
y'(t) =z(t), \quad r(t) =  [-\phi(t)H((N_c*\phi)(t))  +H_*] (N_c*y)(t) \leq 0, \quad t \in \R,
\end{equation}
we find  that 
$$
z'(t)=cz(t) - H_*(N_c*y)(t)+ r(t), \ t \in \R. 
$$
Since $z(\pm\infty)=0$, we also have that 
$$
y'(t)=z(t) = \int_t^{+\infty}e^{c(t-s)}\left(H_*(N_c*y)(s)- r(s)\right)ds
$$
$$
\geq  H_*\int_t^{+\infty}e^{c(t-s)}(N_c*y)(s)ds\geq  H_*\int_t^{+\infty}e^{c(t-s)}y(s)ds\geq H_*y(t)/c, \quad t \in \R. 
$$
Thus 
$$
(y(t)e^{-tH_*/c})'\geq 0, \quad t \in \R, 
$$
which implies (\ref{areI}). \qed
\end{proof} 
We will also need the next property:
\begin{lem}\label{areIi}There exist  $\rho >0$ such  that 
\begin{equation} \label{rh}
(N_c*y)(t) \geq \rho y(t), \quad  t \in \R.
\end{equation}
\end{lem}
\begin{proof}We will distinguish between  two  situations.

\vspace{2mm}

\noindent {\sf Case 1:} supp $N_c \cap (0,+\infty)\not=\emptyset$.  Then there exists $m >0$ such that $\rho_1:= \int_0^mN_c(s)ds >0$ and 
$$
\int_\R y(t-s)N_c(s)ds \geq \int_0^my(t-s)N_c(s)ds \geq  \rho_1 y(t), \quad t \in \R.
$$

\noindent {\sf Case 2:} supp $N_c \cap (0,+\infty)=\emptyset$.  Then, by Lemma \ref{USAI},  we have,  for $t \in \R$,  
$$
(N_c*y)(t) = \int_{-\infty}^0y(t-s)N_c(s)ds \geq  y(t) \int_{-\infty}^0e^{s|H_*|/c}N_c(s)ds=:\rho_2 y(t).
$$
In any event, (\ref{rh}) holds with $\rho \in \{\rho_1,\rho_2\}$. \qed
\end{proof}

\begin{lem}\label{USA} Suppose that $G(0)= \max_{u\geq 0}G(u)$ and, for some $\beta \in (0,1]$, 
\begin{equation}\label{Hi}
H(u)-G'(1) = O((1-u)^\beta), \ u \to 1^-. 
\end{equation}
Let $\phi(t)$ be a monotone wavefront to  equation (\ref{twe2an}). Then there exist $t_1, t_2 \in \R$  such that  
\begin{equation}\label{are}
\phi(t+t_1) = (-t)^je^{\lambda(c)t}(1+o(1)), \ t \to -\infty, \quad \phi(t+t_2) = 1-t^ke^{\hat zt}(1+o(1)), \ t \to +\infty. 
\end{equation}
where $j=0$ if $c >2\sqrt{G(0)}$ and $j=1$ when $c=2\sqrt{G(0)}$; $k \in \{0,1,2,3\}$ and  $\hat z=\hat z(\phi)$ is a negative zero of  the characteristic function 
$\chi_+(z,c)$.   \end{lem}
\begin{proof} \underline{Asymptotic representation of $\phi$ at $+\infty$}. 
Our first step is to establish that $y(t)=1-\phi(t)$ has an exponential rate of convergence to $0$ at $+\infty$.

Since $\phi(+\infty)=1$, we can indicate $T_0$ sufficiently large to satisfy
$$
-\phi(t)H((N_c*\phi)(t)) > -0.5\,G'(1), \quad t \geq T_0. 
$$ 
With the positive number  
$
\kappa = -0.5\,\rho\, G'(1),
$
we can rewrite equation (\ref{efy}) as 
$$
y''(t)-cy'(t) - \kappa y(t) = h(t), \quad \mbox{where} \ h(t):= -\phi(t)H((N_c*\phi)(t))(N_c*y)(t)  - \kappa y(t) , \ t \in \R. 
$$
Importantly, for $t \geq T_0$, 
$$
h(t) > (N_c*y)(t)\left(-0.5\,G'(1) - \frac{\kappa}{\rho}\right)=0.
$$
Next, similarly to (\ref{bl}) (see also \cite[Lemma 20, Claim I]{GT} for more detail), we obtain 
$$
y'(t) -l y(t) = - \int_t^{+\infty}e^{m
(t-s)}h(s)ds < 0, \quad t \geq T_0, 
$$
where $l <0< m$ are the roots of the characteristic equation $z^2 -cz -\kappa =0$. Thus 
$$
y(t) \leq y(s)e^{l(t-s)}, \quad t \geq s \geq T_0, \quad \mbox{where} \ \alpha(y) \leq l= 0.5\left(c-\sqrt{c^2+4\kappa}\right)<0.  
$$
{Hence, by Remark \ref{Re1},  $\frak{I}_{c}(l)$ is a finite number.
Combining the latter exponential  estimate with the results of  Lemma \ref{USAI} (if supp $N_c \cap (0,+\infty)=\emptyset$) or inequality (\ref{rha}) (if supp $N_c \cap (0,+\infty)\not=\emptyset$), we conclude that $y(t)$ has an exponential rate of convergence at $+\infty$. } Moreover,  the same is true for  $y'(t)$ because of the following estimates
$$R(t):= -\left[\phi(t)H((N_c*\phi)(t))\right](N_c*y)(t) \leq  |H_*|(N_c*y)(t) =  |H_*|\int_{-\infty}^{t-T_0} N_c(s)y(t-s)ds
$$
$$+ |H_*|\int^{+\infty}_{t-T_0} N_c(s)y(t-s)ds \leq  |H_*|\int_{-\infty}^{t-T_0} N_c(s)y(T_0)e^{l(t-s-T_0)}ds+ |H_*|\int^{+\infty}_{t-T_0} N_c(s)e^{-ls}e^{ls}ds  $$
$$ \leq  e^{l(t-T_0)}|H_*| \int_{-\infty}^{t-T_0} N_c(s)e^{-ls}ds+  e^{l(t-T_0)} |H_*|\int^{+\infty}_{t-T_0} N_c(s)e^{-ls}ds =  e^{l(t-T_0)}|H_*|\frak{I}_{c}(l)
$$
and 
$$
y'(t)= -\int_t^{+\infty}e^{c(t-s)}R(s)ds\geq -\frak{I}_{c}(l)e^{-lT_0}|H_*|\int_t^{+\infty}e^{c(t-s)}e^{ls}ds= \frac{\frak{I}_{c}(l)e^{-lT_0}|H_*|}{l-c}e^{lt}.
$$
The latter representation of $y'(t)$ is deduced from (\ref{efy}) which also implies that 
\begin{equation}\label{efyk}
y''(t)-cy'(t) +\left(G'(1)+\epsilon(t)\right){(N_c*y)(t)}=0, \quad t \in \R,
\end{equation}
where 
$
\epsilon(t):= \left[\phi(t)H((N_c*\phi)(t))\right] - G'(1). 
$
By  (\ref{Hi}),  
$$\epsilon(t) = (H((N_c*\phi)(t)) - G'(1)) - H((N_c*\phi)(t))y(t)
= O(e^{l\beta t}), \ t \to +\infty.  
$$
Then, in view of Remark \ref{Re1}, an application of \cite[Lemma 22]{TAT} shows that $y(t) = w_0(t)(1+o(1)),$ \  $t \to +\infty$,  where $w_0(t)$ is a non-zero eigensolution of the equation 
$w''(t)-cw'(t) + G'(1)(N_c*w)(t)=0$ corresponding to some its negative eigenvalue $\hat z$. As we have already mentioned, the multiplicity of $\hat z$ is less or equal to 4. This proves the second representation in (\ref{are}). 

\vspace{2mm}

 \noindent \underline{Asymptotic representation of $\phi$ at $-\infty$}.  Since the linear equation $y''-cy'+G(0)y=0$ with $c \geq 2\sqrt{G(0)}$ is exponentially unstable,  so is the following equation 
$$
\phi''(t) - c\phi'(t) + (G(0)+\hat R(t))\phi(t)=0,  \quad t \in \R, 
$$
where 
$
\hat R(t) = {G((N_c*\phi)(t))}-G(0) 
$
and $\hat R(-\infty)=0.$
  
  This assures at least the exponential rate of convergence of $\phi(t), \phi'(t)$ to $0$ at $-\infty$.  On the other hand, $\phi(t), \phi'(t)$ has no more than exponential rate of decay 
 at $-\infty$, cf.  \cite[Lemma 6]{ST}.  Again, an application of \cite[Lemma 22]{TAT} shows that $y(t) = v_0(t)(1+o(1)),$  \ $ t \to +\infty$,  where $v_0(t)$ is the non-zero eigensolution of the equation 
$v''(t)-cv'(t) +G(0)v(t)=0$ corresponding to one of the positive eigenvalues {$\lambda(c), \mu(c)$}. Finally, since the function $F(u,v) = u\, G(v)$  satisfies the sub-tangency condition at   zero equilibrium (this assures that $\hat R(t) \leq 0$, $t \in \R$), we conclude that  the correct eigenvalue in our case is precisely {$\lambda(c)$}, see \cite[Section 7]{GT}  for the related computations and further details. \qed
 \end{proof} 
\begin{cor}  \label{Com} Suppose that $G$ is a strictly decreasing function satisfying (\ref{Hi}). Let $\psi(t), \phi(t)$ be {different} monotone wavefronts to   (\ref{17nl}). 
Then there exist $t_3, t_4 { \in \R}$ such that \quad $\psi(t+t_3) \not= \phi(t+t_4)$ for all $t \in \R,$ meanwhile $\phi(t+t_3),$ $\psi(t+t_4)$ have equal principal  asymptotic terms at $-\infty$:
$$
\phi(t+t_3) = (-t)^je^{{\lambda(c)}t}(1+o(1)), \ \psi(t+t_4) = (-t)^je^{{\lambda(c)}t}(1+o(1)), \ t \to -\infty. 
$$
\end{cor}
\begin{proof} By Lemma \ref{USA}, there are $t_1, t_2, t_1', t_2', q_1 \in \{0,1,2,3\},$ such that (\ref{are}) holds together with 
\begin{equation}\label{arek}
\psi(t+t_1') = (-t)^je^{{\lambda(c)}t}(1+o(1)), \ t \to -\infty, \ \psi(t+t_2') = 1-t^{q_1}e^{{\hat{z}_1}t}(1+o(1)), \ t \to +\infty,
\end{equation}
where ${\hat{z}_1=\hat{z}_1(\psi)}$ is a negative root of the characteristic equation $\chi_+(z,c)=0$. After realizing appropriate translations of profiles, 
without loss of generality, we can assume that $t_1=t_1'=0$. 

Suppose that $\hat z > {\hat{z}_1}$ or $k \geq q_1$ if $\hat z= {\hat{z}_1}$, then there is a sufficiently large  $B>0$ such that 
$$
\psi(t+B) > \phi(t), \quad t \in \R. 
$$
Since $\psi(t)$ is an increasing function,  for all $s \geq B$, 
$$
\psi(t+s) > \phi(t), \quad t \in \R. 
$$
Set $\frak A:= \{s: \psi(t+s) > \phi(t), \  t \in \R \}$.  
Clearly, $\frak A$ is a below bounded set and therefore the number  $s_*= \inf \frak A$
is finite and 
$$
\psi(t+s_*) \geq \phi(t), \quad t \in \R. 
$$
In what follows, to simplify the notation, we suppose that $s_*=0$.  
Observe that,  since $\psi, \phi$ are different wavefronts, the difference $\delta(t)= \psi(t)-\phi(t)$ is a non-zero non-negative function satisfying 
$\delta(-\infty)=\delta(+\infty) =0$. 
We claim that actually $\delta(t) >0, \ t \in \R$, i.e. 
\begin{equation}\label{IN}
\psi(t) > \phi(t), \quad t \in \R. 
\end{equation}
Indeed, otherwise there exists some $s_1 \in \R$ such that $\delta(s_1) =0$. With  the notation
$$
({\mathcal F}f)(t) := [G(0)-G((N_c*f)(t))]f(t)
$$
we have that 
$$
 \delta''(t) - c\delta'(t) +G(0)\delta(t) =  ({\mathcal F}\psi)(t) - ({\mathcal F}\phi)(t) \geq 0, \quad t \in \R,
$$
since 
$({\mathcal F}f)(t)$ 
is an operator non-decreasing  in $f.$

Therefore 
$$
\delta(t) = \int_0^{+\infty}A(s)  \left[({\mathcal F}\psi)(t+s) - ({\mathcal F}\phi)(t+s)\right]ds \geq 0, \quad t \in \R. 
$$
{where $A$ is defined in Section \ref{S22}.}
Since $A(s) >0$ for $s >0$ and $\delta({s_1})=0$, we get immediately that $({\mathcal F}\psi)(s) = ({\mathcal F}\phi)(s)$ for all $s\geq {s_1}$. Clearly, {since $G(u)$ is a strictly decreasing function} this means that 
$$
 \psi(t) = \phi(t) \ (\mbox{i.e.}\ \delta(t)=0), \quad (N_c*\psi)(t) = (N_c*\phi)(t), \quad t \geq {s_1}.
$$
Next, suppose that supp $N_c \cap (0,+\infty)\not=\emptyset$ and that 
$[s',+\infty),$ $s'\in \R$, is the maximal interval where $\delta(t)=0$. Then $\int_{s'-s}^{+\infty}\delta(u)du >0$ for every $s>0$.  Furthermore, since 
$$\int_\R\delta(t-s)N_c(s)ds = (\delta*N_c)(t)=0, \quad t\geq s',$$
we obtain the following contradiction: 
$$
0= \int_{s'}^{+\infty}\left(\int_0^{+\infty}\delta(t-s)N_c(s)ds\right)dt =  \int_0^{+\infty}N_c(s)\left(\int_{s'-s}^{+\infty}\delta(u)du\right)ds >0,
$$
Thus $s'=-\infty$ and $\psi(t) = \phi(t)$ for all $t\in \R$ contradicting  our initial assumption that $\phi$ and $\psi$ are different wavefronts. This proves (\ref{IN}) when 
supp $N_c \cap (0,+\infty)\not=\emptyset$. 

We will use another method when supp $N_c \subseteq  (-\infty,0]$.  In such a case, both functions $v_1(t):= 1-\phi(-t)$ and $v_2(t):=1- \psi(-t)$ solve
the  initial value problem 
\begin{equation}\label{ud}
 v''(t) + cv'(t) -(1-v(t))G\left(1-\int_{-\infty}^0v(t+s)N_c(s)ds\right)=0, 
\end{equation}
$$
 v(\sigma-s') = 1-\phi(s'-\sigma)= 1- \psi(s'-\sigma), \ \sigma \leq 0,\quad v'(-s') = \phi'(s') = \psi'(s').$$
Due to the optimal nature of $s'$, the solutions  $v_1(t)$ and $v_2(t)$ do not coincide on the intervals $(-s', -s'+\epsilon)$ for $\epsilon >0$. 
On the other hand, since the function $g(x,y)= (1-x)G(1-y)$ is globally Lipschitzian  
in the square $[0,1]^2\subseteq \R^2$, we can use the standard 
argumentation\footnote{It suffices to rewrite (\ref{ud}) in an equivalent form of a system of integral equations and then, after some elementary transformations,  to apply the Gronwall-Bellman inequality.} to prove that, for all sufficiently small $\epsilon>0$,  $v_1(t)=v_2(t)$ for  $t\in (-s', -s'+\epsilon)$. Thus again we get a contradiction proving (\ref{IN}) when 
supp $N_c \cap (0,+\infty)=\emptyset$. 

\vspace{2mm}

Next, clearly, 
$$
\kappa_*:=\lim_{t\to -\infty}\psi(t)/\phi(t) \geq 1, \quad \lim_{t\to +\infty}(1-\psi(t))/(1-\phi(t)) \leq 1.
$$
If $\kappa_*=1$, then $\psi(t)$ and $\phi(t)$ have the same asymptotic behavior at $-\infty$ and the {Corollary  \ref{Com} is proved, taking $t_3 = t_4 = 0$}. 

So, let us suppose  that  $\kappa_* >1$. 
Then the optimal nature of $s_*=0$ implies that $\psi(t)$ and $\phi(t)$ have the same asymptotic behavior at $+\infty$. Thus  $\hat z = {\hat{z}_1},$ $ k=q$ and 
$$
\lim_{t\to +\infty}{(1-\psi(t))}/{(1-\phi(t))} =  1.  
$$
But then, for all sufficiently large positive {$D$,}  
$$
\phi(t+ {D}) > \psi(t), \ t \in \R. 
$$
We can now argue as before to establish the existence of the minimal positive $d_*$ such that 
$$
\phi(t+d_*) > \psi(t), \ t \in \R. 
$$
{Then the  optimal character of $d_*$  implies  that 
  $$
\lim_{t\to -\infty}\psi(t)/\phi(t+d_*) =  1. 
$$
Since $d_*>0$ we also have that
$$
 \lim_{t\to +\infty}(1-\phi(t+d_*))/(1-\psi(t)))=  e^{\hat zd_*} <1.
$$
}
This completes the proof of Corollary   \ref{Com} (where {$t_3 =d_*$ and  $t_4 = 0$ 
should be taken)}  in the case $\kappa_* >1$.\qed
 \end{proof}
 

\subsubsection{Proof of Theorem \ref{T12}}
Suppose that there are two different wavefronts, $\phi(t)$ and $\psi(t)$ to equation (\ref{twe2an}).  By Corollary  \ref{Com}, without restricting the generality, we can 
assume that, for each $z \in \R$, 
$$
w(t):=(\psi(t)-\phi(t))e^{-z t} >0, \ t \in \R, \quad  w(t) = e^{({\mu(c)-z})t}(A+o(1)),\  t \to -\infty.
$$
Take now some $z \in (\lambda(c),\mu(c))$ if $c >2\sqrt{G(0)}$ and $z=\lambda(c)=\mu(c)=\sqrt{G(0)}$ if $c=2\sqrt{G(0)}$. Then $w(t)$ is bounded on $\R$ and satisfies
the following equation for all $t \in \R$:
\begin{equation}\label{38}
w''(t)-(c-2z)w'(t)+(z^2-cz+G(0))w(t) = e^{-{z} t}\left( ({\mathcal F}\psi)(t) -  ({\mathcal F}\phi)(t) \right).
\end{equation}
Next, if $c >2\sqrt{G(0)}$, then ${\mu(c)-z} >0$ and therefore $w(-\infty)=w(+\infty)=0$. This means that, for some $t^*$, 
$$
w(t^*)= \max_{s \in \R}w(s) >0, \ w''(t^*) \leq 0, \ w'(t^*) =0. 
$$
Then, evaluating (\ref{38}) at $t^*$ and noting that $z^2-cz+G(0) <0$,  $({\mathcal F}\psi)(t) >  ({\mathcal F}\phi)(t),$ $t \in \R$,  we get a contradiction in signs. 
This proves the uniqueness of all non-critical wavefronts. 

Suppose now that $c=2\sqrt{G(0)}$, then the equation (\ref{38})  takes the form 
$$
w''(t) =e^{-z t}\left( ({\mathcal F}\psi)(t) -  ({\mathcal F}\phi)(t) \right) >0,\quad  t \in \R.
$$
Since $w(+\infty)=0$, this implies that $w'(t)<0$ for all $t \in \R$. Clearly, the inequalities $w'(t)<0, \ w''(t) >0$, $t \in \R$, are not compatible with the boundedness of  $w(t)$ at $-\infty$. This proves the uniqueness of the minimal wavefront. 


\section{On the existence of non-monotone and non-oscillating wavefronts}
The main working tool in this section is  the singular perturbation theory developed by Faria {\it et al} in \cite{fhw,FTnl}. More specifically, we will invoke several results from \cite{FTnl}. For the reader's convenience, they are resumed as  Theorem \ref{TF} in the Appendix.

\subsection{Nonlocal food-limited model with a weak generic delay kernel: proof of Theorem \ref{T20IN}}
Here, following \cite{GG,GC,OW,TPT,WL}, we study the non-local food-limited model 
(\ref{i1IN})
with the  so-called weak generic delay kernel 
(\ref{kernelIN}). 
As in \cite{GC}, 
after introducing the function $v(t,x)= (K *u)(t,x)$, 
we rewrite the model  (\ref{i1IN}), (\ref{kernelIN})  as the system of two coupled reaction-diffusion equations
\begin{equation}\label{s1}
u_t=u_{xx} +u\left(\frac{1-v}{1+\gamma v}\right), \quad v_t=v_{xx} +\frac{1}{\tau}(u-v). 
\end{equation}
Then the task of determining  semi-wavefronts $u(t,x) = \varphi(x+ct)$ to  (\ref{i1IN}), (\ref{kernelIN}) is equivalent to the problem of finding wave solutions 
\begin{equation*}\label{epf}
u(t,x) = \phi(\sqrt{\epsilon} x + t), \quad v(t,x) = \psi(\sqrt{\epsilon} x + t), \quad \epsilon = c^{-2}, 
\end{equation*}
for the system (\ref{s1}). The profiles $\phi, \psi$ satisfy the equations
\begin{equation}\label{ch}
\epsilon \phi'' -\phi' + \phi\frac{1-\psi}{1+\gamma\psi}=0, \quad \epsilon \psi'' -\psi' + \frac{1}{\tau}\left(\phi-\psi\right)=0.
\end{equation}
Note that the characteristic equation for (\ref{ch}) at the positive equilibrium $\phi=1, \psi =1$ is 
\begin{equation}\label{chare}
(\epsilon z^2-z)^2-\frac 1\tau (\epsilon z^2-z)+ \frac{1}{\tau(1+\gamma)}=0.
\end{equation}
If $\tau < (1+\gamma)/4$, it has exactly two positive and two negative simple roots. On the other hand, if $\tau > (1+\gamma)/4$, then it has exactly two complex roots with  positive real parts and two complex roots with  negative real parts. This circumstance explains the necessity of the assumption $\tau \leq (1+\gamma)/4$ for the existence of monotone wavefronts. 

Since we are interested in the positive solutions $(\phi,\psi)$ of (\ref{ch}), we can introduce new variable $\eta$ by $\phi = e^{-\eta}$.  Then (\ref{ch}) can be written as 
 \begin{equation}\label{cs}
\begin{array}{cccc}
\epsilon \psi'  &= & \xi, \\  
\xi' &=& \xi/\epsilon  +\frac 1 \tau (\psi - e^{-\eta}), \\
 \epsilon \eta'& =& \zeta,\\
\zeta' &=&  \zeta/ \epsilon  +  (\zeta/ \epsilon)^2 + \frac{1-\psi}{1+\gamma\psi}. 
\end{array} 
\end{equation}
This system belongs to the class  of monotone cyclic feedback systems (i.e. inequalities (1.10) in \cite{MPSe} are satisfied for (\ref{cs})).  If  $(\phi(t),\psi(t))$  is the wave profile, 
the corresponding solution $\Gamma(t):= (\psi(t),\xi(t), \eta(t), \zeta(t))$ of  (\ref{cs}) is clearly bounded on $\R_+$. Then, in view of studies realized in \cite{El},  we can apply the Main Theorem 
in \cite{MPSm} to conclude that the omega limit set $\omega(\Gamma)$  for $\Gamma(t)$ is either the equilibrium $e:= (1,0,0,0)$ or a nontrivial  periodic orbit
(by  \cite{MPSm},  $\omega(\Gamma)$  cannot contain any orbit homoclinic to $e$ since $\Delta \mbox{det} (-Df(e)) =  -1/(\tau(1+\gamma))$ is negative). This proves the statement of Theorem \ref{T20IN} concerning the asymptotic  shape of the profile $\phi_c(t)$.  
\begin{lem}\label{GST} The positive equilibrium $(1,1)$ of the system 
\begin{equation}\label{Flow}
\phi'(t) = \phi\frac{1-\psi}{1+\gamma\psi}; \quad \psi'(t) = \frac 1 \tau (\phi-\psi),
\end{equation}
is locally exponentially stable and it is also globally stable in the set $\mathcal{Q} =\{ \phi > 0, \psi \geq 0\}$.  The zero equilibrium is a saddle point: the tangent direction  at the origin of the unstable [respectively, stable]  manifold is $( 1+\tau,1)$ [respectively,  $(0,1)$]. Hence, for each fixed pair of parameters $\tau, \gamma$
there exists a unique orbit connecting equilibria $(0,0)$ and $(1,1)$.  Furthermore,  if $\tau < (1+\gamma)/4$ then  the positive equilibrium is a stable node, and all positive semi-orbits, with the  only exception of two trajectories, enter this equilibria in the directions
$$\pm n_1:= \pm\left(1,\frac{1+\gamma}{2\tau}- \sqrt{\frac{(1+\gamma)^2}{4\tau^2}-\frac{1+\gamma}{\tau}}\right).$$
The two above mentioned exceptional trajectories enter $(1,1)$ in the directions 
$$\pm n_2:= \pm\left(1,\frac{1+\gamma}{2\tau}+\sqrt{\frac{(1+\gamma)^2}{4\tau^2}-\frac{1+\gamma}{\tau}}\right).$$
Furthermore, if $\gamma >1, \tau < (1+\gamma)/4,$ then the trajectories of (\ref{Flow}) cannot cross the half-line 
\begin{equation}\label{ln}
\psi := \psi_r(\phi)= [(1+\gamma)/(2\tau)](\phi-1) +1, \ \phi >1,
\end{equation}
from  right  to the left.  

If $\tau > (1+\gamma)/4$ then  the positive equilibrium is a stable focus: in particular,  the heteroclinic solution  spirals into $(1,1)$.  
 \end{lem}
 Observe that  the exceptional direction $n_2$ is "steeper" than $n_1$ and both of them are "steeper" than the diagonal direction $(1,1)$.  The half-line (\ref{ln}) is located in between the half-lines passing trough the point $(1,1)$ in the directions $n_1$ and $n_2$. 
\begin{proof} We begin by noting that the right-hand side of the system (\ref{Flow}) is $C^\infty$-smooth on $\R^2_+$, where it has at most linear growth with respect to $(\phi, \psi)$. 
In addition,  $\R_+^2$ is positively invariant with the respect to  (\ref{Flow}). Indeed, the semi-axis $\phi =0, \psi \geq 0$ is a union of 
the positive half of the stable manifold of the equilibrium $(0,0)$ with  $(0,0)$.    On the other hand, the vector field on the horizontal semi-axis has inward orientation.  
Therefore  (\ref{Flow})  
defines a smooth semi-flow on $\R_+^2$.
 
The characteristic polynomials at the equilibria $(0,0)$ and $(1,1)$ are, respectively, 
$z^2-(1-\tau^{-1})z-\tau^{-1}$ and $z^2+z\tau^{-1}+ (\tau(1+\gamma))^{-1},$ from which we obtain the above mentioned stability properties of both equilibria.  The statement concerning the directions of the integral curves for (\ref{Flow}) at the equilibrium $(1,1)$ follows from a variant of the  Hartman $C^1-$linearization theorem 
 for smooth autonomous systems in a neighborhood of a hyperbolic attractive point, see \cite[p.127]{PH}.    
The computation of the indicated  directions of tangencies is straightforward and it is omitted here. Similarly, the above mentioned property of $\psi_r$ amounts to the inequality 
$$
\frac{1+\gamma}{2\tau} <  \frac{(\phi-\psi_r(\phi))(1+\gamma\psi_r(\phi))}{\tau\phi(1-\psi_r(\phi))}, \quad \phi >1, 
$$
which can be easily checked.  
 
Next, consider the following Lyapunov function 
$$
V(\phi,\psi)=\int_1^\phi \frac{x-1}{x}dx + \tau \int_1^\psi \frac{y-1}{1+\gamma y}dy, \quad \phi, \psi >0. 
$$
It is easy to see that $V$ vanishes at the positive equilibrium only. Calculating the derivative $ \dot{V}$ of $V$ along the trajectories of (\ref{Flow}), we get 
$$
 \dot{V}=-\frac{(\psi-1)^2}{1+\gamma \psi} \leq 0. 
$$
Since the set $\{(\phi,\psi): \dot{V} =0\}\setminus(1,1) = \{\psi =1,\ \phi \geq 0\}\setminus(1,1)$ does not contain an entire orbit of (\ref{Flow}), 
the positive equilibrium is globally asymptotically stable,  see e.g. \cite[Theorem 2, p. 196]{HiS}.  \qed 
\end{proof}  
\begin{figure}[h]
\centering \fbox{\includegraphics[width=9cm]{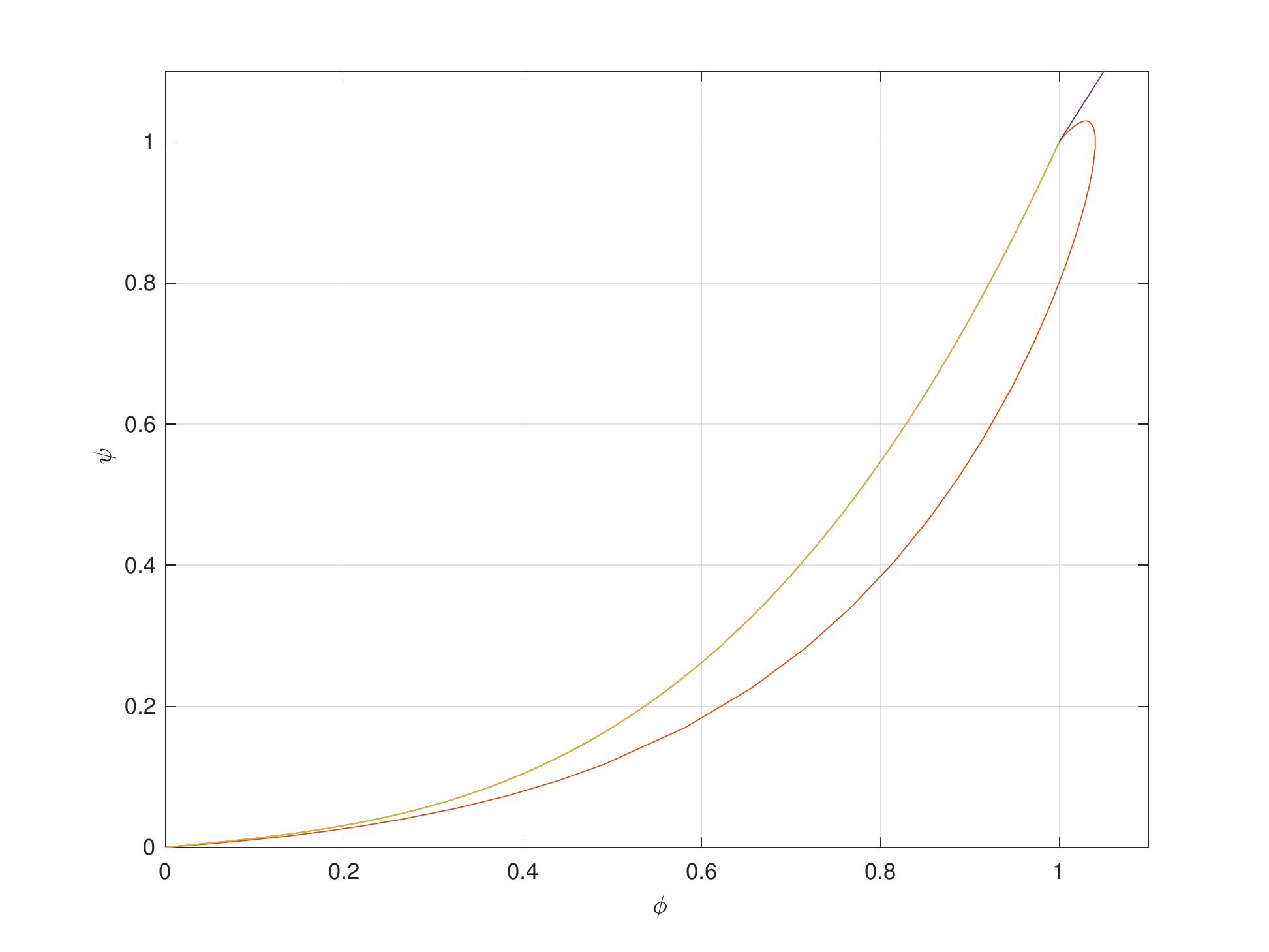}}
\caption{\hspace{0cm} Non-monotone non-oscillating heteroclinic orbit (in red color) and the bounding curves $\psi(\phi), \psi_r(\phi)$ for $\tau=10, \ \gamma = 40$.}
\end{figure}
\begin{lem}\label{NM} Assume that  $\tau \leq  (\gamma+1)/4$ and $\gamma >1$. Suppose that there exists a smooth monotone function  $\psi= \psi(\phi), \ \phi \in [0,1]$ such that 
$\psi(0)=0, \ \psi(1) =1$ and 
\begin{equation}\label{fin}
\psi'(0) > \frac{1}{1+\tau}; \quad \psi'(1) > \frac{1+\gamma}{2\tau}+ \sqrt{\frac{(1+\gamma)^2}{4\tau^2}-\frac{1+\gamma}{\tau}},  
\end{equation}
 \begin{equation}\label{psT}
 \psi'(\phi) > \frac{(\phi-\psi(\phi))(1+\gamma\psi(\phi))}{\tau\phi(1-\psi(\phi))}, \quad \phi \in (0,1).
\end{equation}
Then each component of the heteroclinic
solution  $(\phi(t), \psi(t))$ to the system (\ref{Flow}) is a non-monotone and non-oscillating function  with  exactly one critical point, where the global maximum (bigger than $1$) is reached, see Figure 4.  
 \end{lem}
 \begin{proof} Indeed, inequality (\ref{psT}) implies that the positive orbits of system (\ref{Flow}) starting below the arc $\psi= \psi(\phi), \ \phi \in [0,1],$ cannot cross it in the direction from  right to  left .  The first inequality in (\ref{fin}) also shows that the unstable manifold of the zero equilibrium  in the first quarter lies below the graph of $\psi(\phi)$.  Then the second inequality in (\ref{fin}) as well as the properties of the half-line $\psi_r(\phi)$  in Lemma \ref{GST} oblige the heteroclinic trajectory to approach the equilibrium $(1,1)$  in the direction $-n_1$. 
 The existence of exactly one critical point for each component of the heteroclinic connection follows immediately from the elementary analysis of the vector field near $(1,1)$. \qed
 \end{proof}
 The simplest candidate for the test function $\psi(\phi)$ in Lemma \ref{NM} is the polynomial $\psi = a\phi + b\phi^n,$ where  $b = 1-a$ and  $n \in \N$. 
 Taking $n=3$ 
  we obtain the following
 \begin{cor} \label{Coro} The conclusion of Lemma \ref{NM} holds  true whenever we can find  positive $a, b =1-a$ such that 
 \begin{equation}\label{i2a}
 \frac{1}{\tau+1} < a < 1.5 -  \frac{1+\gamma}{4\tau}- \sqrt{\frac{(1+\gamma)^2}{16\tau^2}-\frac{1+\gamma}{4\tau}},
\end{equation}
 \begin{equation}\label{inok}
\tau(a+3bx^2)(1+bx+bx^2)>b(1+x)(1+\gamma(ax+bx^3)), \quad x \in (0,1). 
\end{equation}
 \end{cor}
Note that the latter inequality can be rewritten as  $P_4(x) >0$, where  $P_4$ is a  real polynomial of order 4.  Since all zeros and critical points 
 of $P_4$ can be calculated explicitly,   inequality  (\ref{inok})  admits a rigorous verification for each fixed set of parameters $a,\tau, \gamma$. 
 
 \begin{example}  For  the parameters $\tau =10, \gamma =40$, numerical simulations in \cite{TPT} suggest the existence of non-monotone and non-oscillating wavefront propagating with speed $c=2$. This numerical result is in  good agreement with Theorem \ref{T20IN}. Indeed, if we take  $\tau=10, \ \gamma = 40$,  then Corollary \ref{Coro} applies with $a=0.12$, see Figure 1.   In fact, for $\gamma =40$, the numerical $\tau-$interval for the existence of non-monotone non-oscillating heteroclinics for (\ref{Flow}) is $(\tau_\#(40), 10.25) := (8.7\dots, 10.25)$ while a shorter interval 
 $(\tau_*(40), 10.25) : = (9.4\dots, 10.25)$  is provided by Corollary \ref{Coro}. 
 \end{example}
 \begin{cor} \label{CoroG} For each $\gamma > 7.3$ there exists $1< \tau_0(\gamma) < (\gamma +1)/4=:\tau_1(\gamma)$ such that the unique heteroclinic connection for the system (\ref{Flow}) is  non-monotone and non-oscillating  for every $\tau \in (\tau_0(\gamma), \tau_1(\gamma)$). 
 \end{cor}
 \begin{proof} For a fixed $\gamma >3$ and positive parameter $\epsilon$, we 
 take 
 $$
 a= \frac 1 2 - \epsilon, \quad \frac{\tau}{1+\gamma} =\frac 1 4 - \epsilon^3.
 $$
 Suppose that $\epsilon>0$ is  small enough to assure the inequalities
 $$
 \tau = (1+\gamma)\left(\frac 1 4 - \epsilon^3\right) >\frac{1+2\epsilon}{1-2\epsilon}, \quad \frac 1 4 - \epsilon^3 > \max\left(\frac{1+2\epsilon^{3/2}}{4(1+\epsilon)}, \frac{1+2\epsilon}{4(1+\epsilon)^2}\right) = \frac 1 4 - \epsilon^2+ \dots. 
 $$
 Then it is easy to see that (\ref{i2a}) is satisfied for all small $\epsilon>0$  as well as the inequality (\ref{inok}) at the end points $x=0,1$. In order to analyze (\ref{inok}) for 
 $x \in (0,1)$, it is convenient to rewrite it in the  finite form 
 $
Q_\epsilon(x)= \sum_{j}\epsilon^j Q_j(x)  \geq 0,   
 $
 where $Q_j(x)$ are certain polynomials, which can be easily calculated. In particular, 
 $$
 Q_0(x) = (x-1)\left((3-\gamma)(x^3+2x^2+2)+(\gamma+13)x\right) >0, \ x \in (0,1), 
  $$
 whenever 
 $$
 \frac{\gamma+13}{\gamma-3} < \min_{x \in (0,1)} \left\{x^2+2x+\frac 2x\right\} = 4.729\dots \quad \mbox{or, equivalently,}\ \gamma > 7.2907 \dots 
 $$
 Observe that $Q_0(1) =0$ and therefore an additional analysis is required to prove the positivity of  $Q_\epsilon(x)$ on $(0,1)$.  We have that  $Q_\epsilon(1)>0$ {for all sufficiently small $\epsilon>0$}  and  $Q_0'(1) =4(7-\gamma) <0$. The latter implies that  $Q_0'(x) <0$ for all $x$ from some small interval  $[x_0,1]$, $x_0<1$.  Therefore
  $Q_\epsilon'(x) <0$
 for all sufficiently small $\epsilon >0$ (say, for $\epsilon\in (0,\epsilon_0]$) and  $x \in [x_0,1]$. Hence,  $ Q_\epsilon(x) = Q_\epsilon(1) - \int^1_xQ'_\epsilon(s)ds >0$ 
 for all $\epsilon\in (0,\epsilon_0]$, $x \in [x_0,1]$. Finally, since  
 $Q_0(x) >0$ for  $x \in [0,x_0]$, we conclude that  there exists $\epsilon_1 \in (0,\epsilon_0)$ such that 
  $Q_\epsilon(x) >0$ for all $\epsilon\in (0,\epsilon_1]$, $x \in [0,1]$. This proves (\ref{inok}) and therefore  Corollary \ref{CoroG} follows from Corollary \ref{Coro}. \qed 
 \end{proof}
 
 \begin{remark} It is easy to see that if the assumptions of Corollary \ref{Coro} are satisfied for some triple of parameters $\gamma', \tau', a'$, with 
 $\tau' < (1+\gamma')/4$ then they will be satisfied for all triples $\gamma', \tau, a'$, with the same $\gamma', a'$ and $\tau \in (\tau', (1+\gamma')/4)$. Therefore, for a fixed $\gamma$, 
 the set of all $\tau$ satisfying the assumptions of  Corollary  \ref{Coro} with some adequate $a,$ is a connected interval, say $(\tau_*(\gamma), (1+\gamma)/4$). Similarly,  these assumptions 
 will be satisfied for all $\gamma, \tau', a'$ such that $4\tau'-1 < \gamma <\gamma'$.  In view of Corollary \ref{CoroG}, all this means that $\tau_*(\gamma)$ is a non-decreasing function defined on the maximal interval $(7.29\dots, +\infty)$. 
  \end{remark}
  
\begin{lem}\label{NMS} For each $\gamma >1$ there exists $\tau_\#(\gamma) \leq (1+\gamma)/4$ such that each component of the heteroclinic
solution  $(\phi(t,\tau, \gamma), \psi(t,\tau, \gamma))$ to system (\ref{Flow}) is a non-monotone and non-oscillating function if and only if $\tau \in (\tau_\#(\gamma),(1+\gamma)/4]$. 
The maximal value of the profile $\phi(t,\tau, \gamma)$ increases as $\tau $  increases (for a fixed $\gamma$) or $\gamma$ decreases (for a fixed $\tau$). Furthermore, 
$\tau_\#(\gamma)$ is a  non-decreasing right-continuous function, and $\tau_\#(\gamma)\leq \tau_*(\gamma) < (1+\gamma)/4$ for all $\gamma > 7.29\dots$  See Figure 2 where the graph of $\tau = \tau_\#(\gamma),\ \gamma \in (7.29\dots, 40]$  is calculated  numerically. 
 \end{lem}
 \begin{proof} Suppose that   (\ref{Flow}) has a non-monotone and non-oscillating heteroclinic for some $\tau', \gamma'$.  Let $\psi=\psi(\phi)$ be the representation for
 this heteroclinic on the maximal open interval for $\phi\in (0,\phi_0), \ \phi_0 >1$. Here $\phi_0:= \max\{\phi(t), \ t  \in \R\}$. Then clearly (\ref{psT}) and the first inequality in (\ref{fin}) hold true on $(0,\phi_0)$ for each  $\gamma < \gamma', \tau=\tau'$ or $\tau > \tau',$ $\gamma= \gamma'$. 
Observe that the second inequality in  (\ref{fin})  does not matter since $\psi(1) <1$. This implies the existence and monotonicity of $\tau_\#(\gamma)$ with the above mentioned properties. 
Finally, suppose for a moment that $\tau_\#$ is not right-continuous at some point $\gamma_0$. Then $\tau_\#(\gamma_0^+) > \tau_\#(\gamma_0)$ so that 
for each fixed $\tau \in (\tau_\#(\gamma_0),\tau_\#(\gamma^+_0))$ and $\gamma_n=\gamma_0+1/n$  (\ref{Flow}) has a monotone  heteroclinic $\psi_n(\phi)$.  But then the limit 
function $\lim_{n\to +\infty}\psi_n(\phi)$ gives a monotone heteroclinic connection for the parameters $\gamma_0, \tau > \tau_\#(\gamma_0)$, a contradiction. \qed 
\end{proof}

Now we can complete the proof of Theorem \ref{T20IN}. To establish the existence of wavefronts for (\ref{ch}), it suffices to check that the right-hand side of system (\ref{Flow}) meets the hypotheses  (H1)-(H4) of Theorem \ref{TF} from  Appendix. (H1), (H2)  are obviously verified {with $\kappa =(1,1)$.} (H3) and (H4) follow from  Lemma \ref{GST}. Now, the general oscillatory and  eventual monotonicity  properties of wavefront profiles follow from the related properties of roots to the characteristic equation (\ref{chare}), see the above discussion. Finally, 
take some $\tau \in  (\tau_\#(\gamma),(1+\gamma)/4]$ for $\gamma > 7.29\dots$. From Corollary \ref{CoroG} and Lemma \ref{NMS} we know that such a pair of $\tau, \gamma$ exists and the associated heteroclinic connection $(\phi_0(t),\psi_0(t))$ of (\ref{Flow}) has non-monotone components. Then Theorem \ref{TF} implies the existence of wavefronts 
with profiles $\phi_0(t,c)$ for all sufficiently large propagation speeds $c$. Since $\phi_0(t,c) \to \phi_0(t)$ uniformly on $\R$ as $c\to +\infty$, these profiles $\phi_0(t,c)$ are non-monotone. However, since $\tau < (1+\gamma)/4$, they also are not oscillating around the level $1$. \qed  

\subsection{Food-limited model with a discrete delay: proof of Theorem \ref{T30IN}}
In this subsection, following \cite{GKL,GG,OW,SY,TPT,WL}, we consider the   diffusive version of the food-limited model with a discrete delay (\ref{i1dIN}). 
Again, looking for wavefronts in the form 
$$
u(t,x) = \phi(\sqrt{\epsilon} x + t), \quad \epsilon = c^{-2}, 
$$
we obtain the profile equation 
\begin{equation}\label{def}
\epsilon \phi''(t) -\phi'(t)+ \phi(t)\left(\frac{1-\phi(t-\tau)}{1+\gamma \phi(t-\tau)}\right)=0. 
\end{equation}
Equation (\ref{def}) with $\gamma =0$ was analyzed  in \cite{BS,P1,ADN,fhw,GT,GTLMS,HTa,HTb,ST,wz}.  The uniqueness of each positive semi-wavefront to equation (\ref{def}) with $\gamma =0$ was proved in \cite{ST}.  Remarkably, the approach of \cite{ST} can be also  applied for $\gamma >0$ since the functional 
$F:C([-\tau,0]) \to \R$ defined as $F(\phi) := \phi(0)(1-\phi(-\tau))/(1+\gamma\phi(-\tau))$ has the following monotonicity property 
$$
F(\psi) - F(\phi) \leq F'(0)(\psi -\phi), \ \mbox{for} \ \phi, \psi \in C([-\tau,0]) \ \mbox{such that} \ 0 \leq \phi(s) \leq \psi(s), \ s \in [-\tau,0]. 
$$
Thus  the uniqueness (up to translation) of each semi-wavefront follows from \cite[Theorem 1]{ST} (see also \cite[Corollary 2]{ST} for computation details  when $\gamma =0$).

Similarly, when $\gamma =0$, it was proved in 
\cite[Section 3]{HTa} that each  positive solution  $\phi(t)$ of (\ref{def}) satisfying $v(-\infty)=0$ is either eventually monotone or it is sine-like slowly oscillating around $1$ at $+\infty$.  
It is easy to see that, due to the strict monotonicity of the function $G(u) = (1-u)/(1+\gamma u)$ on $\R_+$,  the proof given in \cite[Section 3]{HTa}  is also valid for the case $\gamma >0$ if we consider an additional possibility  when the solution  $\phi(t)$ slowly oscillates around $1$ on a finite interval and then converges monotonically to $1$. 

Next, the limit form of (\ref{def}) with $\epsilon =0$  is 
\begin{equation}\label{defs}
\phi'(t)= \phi(t)\left(\frac{1-\phi(t-\tau)}{1+\gamma \phi(t-\tau)}\right). 
\end{equation}
It is immediate to see that the hypotheses (H1), (H2), (H4)  of  Theorem \ref{TF} from  Appendix are satisfied with $\kappa=1$ for the equation (\ref{defs}). For $\tau < 1.5(1+\gamma)$, 
the assumption (H3) is also satisfied in view of \cite[Example 1.4]{lprtt}. Hence, by   Theorem \ref{TF},  equation (\ref{defs}) has a positive heteroclinic connection $\phi= u^*(t)$ and 
 equation (\ref{def}) has a positive heteroclinic connection $\phi(t,\epsilon)$ for all small $\epsilon >0$. Moreover, when $\epsilon \to 0^+$ it holds  that $\phi(t,\epsilon) \to u^*(t)$ uniformly on $\R$. A simple  analysis shows that only the following two possibilities can happen for $u^*(t)$: either  it   is strictly monotone on $\R$ taking values in the interval $(0,1)$ or $u^*(t)$  crosses transversally the level $\phi = 1$ at some sequence of points 
$t_1 <\dots < t_n$ where $t_{j+1} > t_j+\tau$ (this sequence can be finite, $t_n \in \R \cup \{\infty\}$). Linearizing (\ref{defs})  at the steady state $\phi(t) =1$, 
we find the associated characteristic equation
$$
z =-\frac{e^{-z\tau}}{1+\gamma}. 
$$
After a straightforward computation, we can see that this equation has exactly two different real roots $z_2 < z_1<0$ if and only if $\tau < (1+\gamma)/e$.  Moreover, under the latter condition, each other root 
$z_j =\alpha_j+i\beta_j,\ \beta_j \geq0, \ j >2$,  satisfies the inequalities $\alpha_j < z_2$ and $\beta_j > 2\pi/\tau$, see \cite[Theorem 6.1]{mp}. We claim that $u^*(t)$  is eventually 
monotone at $+\infty$ if $\tau < (1+\gamma)/e$. Indeed, otherwise, as we have mentioned above, $u^*(t)$ exponentially converges  to the equilibrium  $1$ and  slowly oscillates around it. Then by the Cao theory of super-exponential solutions, see \cite[Theorem 3.4]{Cao},  there exists a root $z_j=\alpha_j+i\beta_j, j \in \N,$ of the characteristic equation and $C\not=0, \ \delta >0, \ \theta \in \R,$ such that $u^*(t)$ admits the asymptotic representation 
$$
u^*(t)-1 = Ce^{\alpha_jt}\cos(\beta_j t+\theta) + O(e^{(\alpha_j-\delta) t}), \quad t \to +\infty. 
$$
However, if $z_j$ is not a real root, i.e. if  $j >2$, the above representation implies that the distance between large adjacent zeros of  $u^*(t)-1$ is less than $\tau/2$, i.e. $u^*(t)$ oscillates rapidly around $1$, a contradiction. Hence\begin{equation}\label{j}
u^*(t)=1+ C_*e^{z_jt} + O(e^{(z_j-\delta) t}), \quad t \to +\infty,
\end{equation}
where  $j \in \{1,2\}, \alpha_j =z_j <0, \  \beta_j =0$, 
and our claim is proved.  

After some non-trivial technical work, the eventual monotonicity  property of $u^*(t)$ can be extended for $\phi(t,\epsilon)$ for all small $\epsilon \geq 0$ (note that the Cao theory cannot be applied  to the second order delay differential  equations): 
\begin{lem} \label{28} Assume that $0< \tau < (1+\gamma)/e$. Then there exists $\epsilon_0>0$ such that the solution $\phi(t,\epsilon)$ is eventually monotone at $+\infty$ for each $\epsilon\in [0,\epsilon_0]. $
\end{lem}
\begin{proof}
After linearizing (\ref{def})  at the equilibrium  $\phi(t) =1$, 
we find the related characteristic equation
\begin{equation}\label{ehe}
\chi(z,\epsilon):= \epsilon z^2 -z -\frac{e^{-z\tau}}{1+\gamma}=0. 
\end{equation}
If  $0< \tau < (1+\gamma)/e$, it follows from \cite[Lemma 1.1]{GTLMS} that there are $\delta >0$ and $\epsilon_1 >0$ such that (\ref{ehe}) has in the half-plane 
$\Re z > z_2-2\delta$ for each $\epsilon \in (0,\epsilon_1]$ exactly three roots $z_j(\epsilon), \ j =0,1,2$.  Moreover, these roots are real and $z_j(\epsilon) \to z_j, \ j =1,2$, and $z_0(\epsilon) \to +\infty$ as $\epsilon \to 0^-$. 
Therefore the constant solution  $\phi(t) =1$ to (\ref{def}) is hyperbolic and the orbit associated with the heteroclinic  $\phi(t,\epsilon)$  belongs to the stable manifold of $\phi(t) =1$. 
This implies that 
$$
\phi(t,\epsilon) = 1+ C_{{j}}(\epsilon)e^{z_j(\epsilon)t}  + O(e^{(z_j-\delta)t}), \quad  t \to +\infty, \quad \epsilon \geq 0, \ \mbox{for some} \  j= j(\epsilon)  \in \{1, 2\}, 
$$
where, similarly to \cite[Section 7]{GT}, $C_j(\epsilon)$ can be calculated as 
$$
C_j(\epsilon)= \mbox{Res}_{z=z_j(\epsilon)}\frac{\tilde R(z,\epsilon)}{\chi(z,\epsilon)} = \frac{\tilde R(z_j(\epsilon),\epsilon)}{\chi'(z_j(\epsilon),\epsilon)} , \quad \tilde R(z,\epsilon):= \int_{\R} R(s,\epsilon)e^{-zs}ds, $$
$$ R(t,\epsilon):= 
\frac{\phi(t-\tau, \epsilon)-1}{1+\gamma}\cdot \frac{(\phi(t, \epsilon)-1) +\gamma(\phi(t, \epsilon)-\phi(t-\tau, \epsilon))}{1+\gamma\phi(t-\tau, \epsilon)}, \quad  \epsilon \geq 0. 
$$
In view of  Remark \ref{re3} and Lemma 4.1 in \cite{ATV} (it is easy to see that we can take $h, p <0$ as well as the negative sign before $y'(t)$ in the referenced  lemma),  $\tilde R(z,\epsilon)$ is a continuous function so that  $C_j(\epsilon)$ depends continuously on $\epsilon$ from some non-degenerate interval  $[0,\epsilon_2] \subset [0,\epsilon_1]$. Hence, if $j=1$ in (\ref{j}) then 
$C_1(\epsilon)\not=0$ for all small $\epsilon$ (say, for $\epsilon \in [0,\epsilon_0]$) and the conclusion of Lemma \ref{28} follows. Suppose now that $C_1(0)=0$ and that the closed set 
$\mathcal E = \{\epsilon \in [0,\epsilon_0]: C_1(0)=0\}$ is infinite and has $0$ as its accumulation point. Then $C_2(0)\not= 0$ so that, without loss of generality, $C_2(\epsilon)\not= 0$ for all $ \epsilon \in {\mathcal E}$. Since 
$$
\phi(t,\epsilon) = 1+ C_2(\epsilon)e^{z_2(\epsilon)t}  + O(e^{(z_2-\delta)t}), \quad  t \to +\infty, \quad \epsilon \in {\mathcal E}, 
$$
$\phi(t,\epsilon)$ is eventually monotone at $+\infty$ for each $\epsilon\in [0,\epsilon_0].$ \qed 
\end{proof}
Somewhat surprising fact is that even if $0< \tau < (1+\gamma)/e$ then $u^*(t)$ can be non-monotone for certain parameters $\tau, \gamma$: 
\begin{lem} \label{30} Assume that $0< \tau < (1+\gamma)/e$ and
\begin{equation}\label{mi}
\zeta := \max_{a\in [0,1]} a \exp\left(\tau + \frac{(1-a)\tau}{1+\gamma a}\right)\left(\frac{1+a\gamma}{1+a\gamma e^{\tau}}\right)^{1+1/\gamma} >1. 
\end{equation}
Then solution $u^*(t) $ is eventually monotone at $+\infty$ and  it is non-monotone on $\R$. 
\end{lem} 
\begin{proof} For a fixed $a \in (0,1)$ there exists a unique point $t_a$ such that 
$u^*(t)$ is strictly increasing on $(-\infty, t_a)$ and $u^*(t_a) =a$. Without loss of generality, we can suppose that $t_a=-\tau$. Then, after integrating (\ref{defs}) on $[-\tau, 0]$ and using the monotonicity of $G(u)$, we find that 
$$
a\exp\left(\frac{(1-a)(t+\tau)}{1+\gamma a}\right) \leq u^*(t) =a\exp\int_{-\tau}^t \left(\frac{1-u^*(s-\tau)}{1+\gamma u^*(s-\tau)}\right)ds\leq ae^{(t+\tau)}, \ t \in [-\tau, 0]. 
$$
Similarly, by integrating (\ref{defs}) on $[0,\tau]$, we obtain  that, for $t \in [0,\tau]$, 
$$
u^*(t) \geq a\exp\left(\frac{(1-a)\tau}{1+\gamma a}\right)\exp \int_{0}^t \left(\frac{1-u^*(s-\tau)}{1+\gamma u^*(s-\tau)}\right)ds\geq $$
$$ a\exp\left(\frac{(1-a)\tau}{1+\gamma a}\right)\exp \int_{0}^t \left(\frac{1-ae^{s}}{1+\gamma ae^{s}}\right)ds=  ae^t\left(\frac{1+a\gamma}{1+a\gamma e^{t}}\right)^{1+1/\gamma}\exp\left(\frac{(1-a)\tau}{1+\gamma a}\right). 
$$
Clearly, the last inequality evaluated at the point $t =\tau$ together with (\ref{mi}) imply the conclusion of the lemma. 
\qed
\end{proof}
This completes the proof of  Theorem \ref{T30IN}.  
The graph of $u^*(t)$ for the parameters $\gamma =9,$ $\tau =3$ is shown on Figure 1 in the Introduction.

\section*{Acknowledgments}  \noindent We would like to thank Zuzana Chladn\'a
 for her computational and graphical work some of which is used in this paper.  
This work was realized during a research stay of S. Trofimchuk at the Silesian University in Opava,  Czech Republic. This stay was possible due to the support of  the Silesian University in Opava and of the European Union through the project  CZ.02.2.69/0.0/0.0/16\_027/0008521.  S. Trofimchuk  was  also partially  supported by FONDECYT (Chile),   project 1190712. The work of Karel Has\'ik, Jana Kopfov\'a and Petra N\'ab\v{e}lkov\'a was supported  by the institutional support
for the development of research organizations I\v C 47813059.

\section*{Appendix}
Theorem \ref{TF} below follows from \cite[Theorems 2.1, 3.8, 4.3; Corollaries 3.9, 3.11]{FTnl}. 
To state this result, we first introduce  some  notation.
For $d=(d_1,\dots ,d_N)\in\mathbb{R}^N$, we say that $d>0$
(respectively $d\ge 0$) if $d_i>0$ (respectively $d_i\ge 0$) for
$i=1,\dots,N$. In the Banach space ${\cal C}=C([-\tau,0];\R^N)$, we consider the  partial orders $\geq,$ resp. $>,$  defined as follows: 
$\phi\ge \psi$ if and only if $\phi(\th)- \psi(\th)\ge 0$ for
$\th\in [-\tau,0]$; in an analogous   way, $\phi >\psi$ if  $\phi(\th)-
\psi(\th)> 0$ for $\th\in [-\tau,0]$. Next, ${\cal C}_+$ denotes
the positive cone $C([-\tau,0];[0,\infty)^N)$.

 We also will need the following Banach spaces:

$C_0=\{ y\in C_b:\lim _{s\to\pm \infty}y(s)=0\}$ is considered as a
subspace of $C_b$;

$C_\mu=\{ y\in C_b:\sup_{s\le 0}e^{-\mu s}|y(s)|<\infty\}$ (for
$\mu>0$) with the norm
\begin{equation*}
\|y\|_\mu =\max \{\|y\|_\infty,\|y\|_{\mu}^- \}\q {\rm where} \q
 \|y\|_\mu^-=\sup_{s\le 0}e^{-\mu s}|y(s)|;
\end{equation*}
The space  $C_{\mu,0}=C_{\mu}\cap C_0$ will be  considered as a subspace of $C_\mu$.

We will analyze certain singular perturbations of the heteroclinic connection in the system of functional differential equations 
\begin{equation}\label{e1.2}
u'(t)=f(u_t),\q t\in \mathbb{R},
\end{equation}
where $f$ is such that
\begin{description}
\item [{(H1)}] $f(0)=f(\kappa)=0$, where $\kappa$ is some positive vector;
\item [{(H2)}] (i) $f$  is $C^2$-smooth; furthermore,
(ii) for every $M>0$ there is $\be >0$ such that $f_i(\var)+\be
\var_i(0)\ge 0,$ $ i=1,\dots,N$, for all $\var\in {\cal C}$ with
$0\le \var\le M$;
\item [{(H3)}] for Eq. (\ref{e1.2}),  the equilibrium $u=\kappa$
is locally asymptotically stable and
 globally attractive in the set of solutions of (1.2)  with the initial conditions $\var\in {\cal C}_+,\var(0)> 0$;
\item  [{(H4)}] for Eq. (\ref{e1.2}), its linearized equation about
the equilibrium 0 has a real characteristic root $\la_0>0$, which is
simple and dominant (i.e., $\Re\, z<\la_0$ for all other
characteristic roots $z$);  moreover,
  there is  a characteristic   eigenvector ${\bf v}>0$ associated  with $\la_0$.
\end{description}
We are now in a position to present the above mentioned perturbation result:  
\begin{thm} \label{TF} \cite{FTnl} Assume (H1)-(H4), then equation (\ref{e1.2}) has a positive heteroclinic solution $u^*(t)$:\ $u^*(-\infty)=0, \  u^*(+\infty)=\kappa$. 
Next, denote by $\sigma(A)$ the set of characteristic values for
$$
u'(t)=Lu_t,\q {\rm where}\q  L=Df(0). 
$$
Let positive $\mu<\la_0 $  be such that  the strip $\{
\la\in \mathbb{C}: \Re\, \la \in (\mu, \la_0)\}$ does not intersect
$\sigma (A)$. Then there exist 
a direct sum representation 
\begin{equation*}
C_{\mu ,0}=X_\mu \oplus Y_\mu, \quad \mbox{where} \ X_\mu, Y_\mu \ \mbox{are subspaces of }\ C_{\mu ,0}, \  \ \dim X_\mu=1, 
\end{equation*} and 
$\vare^*>0$, $\sigma >0$, such that for
$0<\vare\le \vare^*$, the following  holds:
for each unit vector
$w\in \mathbb{R}^p$, in a neighbourhood
$B_\sigma(0)$ of $u^*(t)$ in $C_\mu$, the set of all
wavefronts $u(t,x)=\psi (ct+w\cdot x)$ of 
$$
{{\p u}\over {\p t}}(t,x)=\Delta u(t,x)+f(u_t(\cdot ,x)),\q t\in
\mathbb{R},\  x\in\mathbb{R}^p.
$$
with speed $c=1/\vare$ and connecting $0$ to $\kappa$ forms a  one-dimensional
manifold (which does not depend on the choice of $\mu$), with the profiles
$$\psi(\vare, \xi) =u^*+\xi +\phi (\vare, \xi),\quad  \xi\in X_\mu\cap B_\sigma(0),$$
where  $\phi (\vare, \xi) \in Y_\mu\cap B_\sigma (0)$
 is continuous in $(\vare,\xi)$. Next,  the profile $\psi
(\vare,0)$  is positive and satisfies $
\psi (\vare,0)\to u^*\ {\rm in} \ C_\mu $ as  $\vare\to 0^+
$. Moreover, the components of the profile $\psi(\vare, 0)$ are increasing in the
vicinity of $-\infty$ and  $\psi(\vare,0)(t)=O(e^{\la (\vare)t}),$\ $ 
\psi'(\vare,0)(t)=O(e^{\la (\vare)t})$ at $-\infty$, where $\vare=1/c$ and
$\la (\vare)$ is the real  solution of 
$$
\det \Delta_\vare (z)=0,\q {\rm}\q \Delta_\vare (z):=\vare^2
z^2I-zI+L(e^{z\cdot}I),
$$
where $L=Df(0)$, 
with
$\la(\vare)\to \la_0$ as $\vare\to 0^+$.
\end{thm}
\begin{remark} \label{re3}In fact, after a  slight modification of 
 the proof of Theorem 3.8 in \cite{FTnl}, one can note that the conclusions of Theorem \ref{TF} remain valid if we replace the space $C_{\mu ,0}$ with $C_{\mu,\delta}=\{ y\in C_b:\|y\|_{\mu,\delta} <\infty\}$ for
 small $\delta \in (0,\mu)$ and with the norm
 \begin{equation*}
\|y\|_{\mu,\delta} =\max \{\|y\|_\delta^+,\|y\|_{\mu}^- \}\q {\rm where} \q
 \|y\|_\mu^-=\sup_{s\le 0}e^{-\mu s}|y(s)|, \  \|y\|_\delta^+=\sup_{s\ge 0}e^{\delta s}|y(s)|. 
\end{equation*}
To see this, it suffices to use the change of variables $\phi(t) = u^*(t)+ e^{-\delta t}w(t)$ instead of $\phi(t) = u^*(t)+ w(t)$ before formula (3.3) in \cite{FTnl}. As a consequence of this 
observation, there exist small $\epsilon_0 >0$ and  some constant $C>0$ which does not depend on $\epsilon$ such that 
$|\psi(\vare,0)(t) -\kappa| \leq Ce^{-\delta t}$, $t\geq 0$, \ $\epsilon \in [0,\epsilon_0]$. 
 \end{remark}

\end{document}